\pdfoutput=1
\documentclass[a4paper,11pt,reqno]{amsart}
\usepackage[utf8]{inputenc}
\usepackage[pdftex]{graphicx}
\usepackage{a4wide}
\usepackage[english]{babel}
\usepackage[T1]{fontenc}
\usepackage{amsmath,amssymb,amsthm,stmaryrd}
\usepackage{mathrsfs}
\usepackage{esint}
\usepackage{mathtools}
\usepackage{comment}
\usepackage{color}
\usepackage{extarrows}
\usepackage{tikz}
\usepackage{tcolorbox}
\usepackage{setspace}
\usetikzlibrary{arrows}

\usepackage{pgf}
\usepackage{color}

\DeclareMathOperator*{\fiint}{\ensuremath{\iint\text{\kern-1.36em{\raisebox{5.87pt}{\rotatebox{-93}{$\setminus$}}}}}}

\DeclareMathOperator{\diam}{diam}
\DeclareMathOperator{\dist}{dist}
\DeclareMathOperator{\pr}{pr}
\DeclareMathOperator{\smallt}{t}
\DeclareMathOperator{\HMT}{HMT}
\DeclareMathOperator{\loc}{loc}

\newcommand{\nablat}{\nabla_{\smallt}}
\newcommand{\nablahmt}{\nabla_{\HMT}}
\newcommand{\Lip}{\operatorname{Lip}}

\newcommand{\R}{\mathbb{R}}
\newcommand{\ree}{\mathbb{R}^{n+1}}

\newcommand{\D}{\mathbb{D}}

\newcommand{\N}{\mathbb{N}}
\newcommand{\Z}{\mathbb{Z}}

\newcommand{\Bc}{\mathcal{B}}

\newcommand{\Hc}{\mathcal{H}}
\newcommand{\Cc}{\mathcal{C}}

\newcommand{\Ac}{\mathcal{A}}
\newcommand{\Dc}{\mathcal{D}}
\newcommand{\Ec}{\mathcal{E}}

\newcommand{\Rc}{\mathcal{R}}

\newcommand{\Uc}{\mathcal{U}}

\newcommand{\Cs}{\mathscr{C}}

\newcommand{\util}{\widetilde{u}}

\newcommand{\eps}{\varepsilon}
\newcommand{\pom}{\partial\Omega}

\newtheorem{theorem}[equation]{Theorem}
\newtheorem*{theorem*}{Theorem}
\newtheorem{lemma}[equation]{Lemma}

\newtheorem{corollary}[equation]{Corollary}
\newtheorem{proposition}[equation]{Proposition}

\theoremstyle{definition}
\newtheorem{defin}[equation]{Definition}
\newtheorem*{defin*}{Definition}

\newtheorem*{question*}{Question}
\newtheorem{problem}[equation]{Problem}
\newtheorem*{problem*}{Problem}
\newtheorem{example}[equation]{Example}
\newtheorem{remark}[equation]{Remark}
\newtheorem*{remark*}{Remark}
\newtheorem{notation}[equation]{Notation}

\numberwithin{equation}{section}

\title{Connectivity conditions and boundary Poincar\'e inequalities}

\author{Olli Tapiola and Xavier Tolsa}

\address{Olli Tapiola, Departament de Matemàtiques, Universitat Autònoma de Barcelona, Bar\-ce\-lo\-na, Catalonia.}
\email{olli.m.tapiola@gmail.com}

\address{Xavier Tolsa, ICREA, Barcelona, Departament de Matemàtiques, Universitat Autònoma de Barcelona, and Centre de Recerca Matemàtica, Barcelona, Catalonia.}
\email{xtolsa@mat.uab.cat}

\date{\today}
\thanks{The authors were supported by the European Research Council (ERC) under the European Union's Horizon 2020 research and innovation programme (grant agreement 101018680) and by the Deutsche Forschungsgemeinschaft (DFG, German Research Foundation) under Germany's Excellence Strategy -- EXC-2047/1 -- 390685813. X.T. was also partially supported by MICINN (Spain) under the grant PID2020-114167GB-I00, the María de Maeztu Program for units of excellence (Spain) (CEX2020-001084-M), and 2021-SGR-00071 (Catalonia).}

\begin{document}

\begin{abstract}
  Inspired by recent work of Mourgoglou and the second named author, and earlier work of Hofmann, Mitrea and Taylor, we consider connections between the local John condition, the Harnack chain condition and weak boundary Poincar\'e inequalities in open sets $\Omega \subset \R^{n+1}$, with codimension $1$ Ahlfors--David regular boundaries. First, we prove that if $\Omega$ satisfies both the local John condition and the exterior corkscrew condition, then $\Omega$ also satisfies the Harnack chain condition (and hence, is a chord-arc domain). Second, we show that if $\Omega$ is a $2$-sided chord-arc domain, then the boundary $\pom$ supports a Heinonen--Koskela type weak $1$-Poincar\'e inequality. We also construct an example of a set $\Omega \subset \R^{n+1}$ such that the boundary $\pom$ is Ahlfors--David regular and supports a weak boundary $1$-Poincar\'e inequality but $\Omega$ is not a chord-arc domain. Our proofs utilize significant advances in particularly harmonic measure, uniform rectifiability and metric Poincaré theories.
\end{abstract}

\maketitle

\tableofcontents

\section{Introduction}

Extending the PDE theory from smooth or otherwise ``nice'' domains to spaces with rough geometries has been under intensive research over the previous decades. The new tools and techniques developed by numerous authors have helped to overcome many difficulties in this field but many questions remain open. See \cite{hofmann_survey} and \cite{mattila} for recent surveys related to key developments in some parts of this and related research.

In this paper, we consider connections between John-type conditions, the Harnack chain condition and Poincar\'e inequalities in spaces with rough boundaries, inspired by the recent work of Mourgoglou and the second named author \cite{mourgogloutolsa2}. In their work, they solved a long-standing open problem of Kenig \cite[Problem\,3.2.2]{kenig} (see also \cite[Question\,2.5]{toro}) about the solvability of the regularity problem for the Laplacian. More precisely, they proved the following theorem. We give the exact definitions of the objects and concepts in the theorem in Sections \ref{section:notation} and \ref{section:john_conditions}.

\begin{theorem}[{\cite[part of Theorems 1.2 and 1.5]{mourgogloutolsa2}}]
  \label{theorem:mourgogloutolsa}
  Let $\Omega \subset \R^{n+1}$ be a bounded open set satisfying the corkscrew condition (see Definition \ref{defin:corkscrew}), with $n$-dimensional Ahlfors--David regular boundary $\pom$ (see Definition \ref{defin:adr}). For $1<p\leq2$ we have:  
  \begin{enumerate}
    \item[(a)] The regularity problem for the Laplacian is solvable in $L^p$ for $\Omega$ (see Definition \ref{defin:regularity_problem}) if and only if the Dirichlet problem is solvable in $L^{p'}$ (see Definition \ref{defin:lp-sovability}), where $p'$ satisfies $1/p + 1/p' = 1$.

    \item[(b)] Suppose that either $\pom$ supports a weak $p$-Poincar\'e inequality (see Definition \ref{defin:weak_poincare}) or that $\Omega$ satisfies the $2$-sided local John condition (see Definition \ref{defin:2-sided_conditions} and Definition \ref{defin:john_conditions}). If the regularity problem for the Laplacian is solvable in $L^p$ for $\Omega$, then the tangential regularity problem for the Laplacian is also solvable in $L^p$ for $\Omega$ (see Definition \ref{defin:regularity_problem}).
  \end{enumerate}
\end{theorem}

Remark that, in particular, from (a) in the theorem above it follows that the regularity problem is solvable in $L^p$ for chord-arc domains (see Definition \ref{defin:cad}), for some $p > 1$ (see e.g. \cite{davidjerison} and \cite{semmes_analysis}). This extends previous results by Jerison and Kenig \cite{jerisonkenig1} in the plane, and by Verchota \cite{Verchota} in Lipschitz domains.

Our goal is to revisit the assumptions of Theorem \ref{theorem:mourgogloutolsa} (b) by studying them from the point of view of $2$-sided chord-arc domains: we show that a $2$-sided local John domain with codimension $1$ Ahlfors--David regular boundary is a $2$-sided chord-arc domain, and the boundary of any $2$-sided chord-arc domain supports weak Poincar\'e inequalities. To be more precise, we prove the following two results:

\begin{theorem}
  \label{theorem:local_john_harnack}
  Suppose that $\Omega \subset \R^{n+1}$ is an open set with $n$-dimensional Ahlfors--David regular boundary that satisfies the local John condition and the exterior corkscrew condition (see Definition \ref{defin:2-sided_conditions} and Definition \ref{defin:corkscrew}). Then $\Omega$ also satisfies the Harnack chain condition (see Definition \ref{defin:nta}). In particular, a $2$-sided local John domain with codimension $1$ Ahlfors--David regular boundary is a $2$-sided chord-arc domain.
\end{theorem}

\begin{theorem}
  \label{theorem:boundary_poincare}
  Suppose that $\Omega \subset \R^{n+1}$ is a $2$-sided chord-arc domain. Then the following weak $1$-Poincar\'e inequality for Lipschitz functions on $\pom$ holds: there exist constants $C \ge 1$ and $\Lambda \ge 1$ such that for every Lipschitz function $f$ on $\pom$ and every $\Delta = \Delta(y,r) = B(y,r) \cap \pom$ we have
  \begin{align}
    \label{inequality:tan_poincare} \fint_\Delta |f(x) - \langle f \rangle_\Delta| \, d\sigma(x) \le C r \fint_{\Lambda \Delta} |\nablat f(x)| \, d\sigma(x),
  \end{align}
  where $\nablat f$ is the tangential gradient of $f$ (see Definition \ref{defin:tangential:gradient}), $\sigma \coloneqq \Hc^n|_{\pom}$ is the surface measure and $\langle f \rangle_\Delta \coloneqq \tfrac{1}{\sigma(\Delta)} \int_\Delta f(y) \, d\sigma(y)$ is the integral average of $f$ over $\Delta$.
\end{theorem}

We note that the conclusion of Theorem \ref{theorem:local_john_harnack} holds if the local John condition is ``good enough'' in the sense that a $D_0$-local John condition implies the Harnack chain condition but a $(D_0,R_0)$-local John condition for $R_0 = c \cdot \diam(\pom)$ and $c < 1$ small enough does not, if $\diam(\pom) < \infty$. See the definitions and discussion in Section \ref{section:john_conditions}.

Theorem \ref{theorem:boundary_poincare} and its consequence Corollary \ref{corollary:poincare} improve some results in the literature. Earlier, Semmes proved that a weak $2$-Poincar\'e inequality for the tangential gradient (that is, inequality \eqref{inequality:tan_poincare} with the right hand side replaced by $C r (\fint_{\Lambda \Delta} |\nablat f(x)|^2 \, d\sigma(x))^{1/2}$) holds for smooth functions on any \emph{chord-arc surface with small constant} \cite[Lemma 1.1]{semmes_hypersurfaces} (see the introduction in \cite{semmes_hypersurfaces} for the definition of these surfaces). Theorem \ref{theorem:boundary_poincare} both provides a stronger inequality and generalizes the class of surfaces considered by Semmes. A key element in the proof of Theorem \ref{theorem:boundary_poincare} is the machinery built by Hofmann, Mitrea and Taylor in \cite{hofmannmitreataylor}. In their paper, they prove a weak $(p,p)$-Poincar\'e inequality with a tail for the Hofmann--Mitrea--Taylor Sobolev space $L_1^p(\pom)$ with respect to the Hofmann--Mitrea--Taylor gradient (see Definition \ref{defin:hmt_gradient}) for any $1 < p < \infty$ on boundaries of $2$-sided local John domains \cite[Proposition 4.13]{hofmannmitreataylor}. Combining Theorem \ref{theorem:boundary_poincare} with some density results in \cite{hofmannmitreataylor} and tools in \cite{mourgogloutolsa2} shows us that the tail in their inequality can be removed, at least when $\Omega$ is a bounded $2$-sided chord-arc domain (see Corollary \ref{corollary:poincare_for_hmt_gradient}).

Our results have some immediate consequences. First, we note that since chord-arc domains satisfy the local John condition, Theorem \ref{theorem:local_john_harnack} gives us the following characterization result:

\begin{corollary}
  \label{corollary:local_john_cad}
  Let $\Omega \subset \R^{n+1}$ be an open set satisfying a $2$-sided corkscrew condition, with $n$-dimensional Ahlfors--David regular boundary $\pom$. Then the following conditions are equivalent:
  \begin{enumerate}
    \item[(a)] $\Omega$ satisfies the local John condition,
    \item[(b)] $\Omega$ satisfies the Harnack chain condition.
  \end{enumerate}
\end{corollary}

For recent related results for semi-uniform domains and chord-arc domains, see \cite{azzam_semi-uniform} and \cite{azzametal2}.

Second, Theorem \ref{theorem:boundary_poincare} combined with a Lipschitz characterization of Poincar\'e inequalities \cite[Theorem 2]{keith} gives us the following Heinonen--Koskela type weak $1$-Poincar\'e inequality:

\begin{corollary}
  \label{corollary:poincare}
  Let $\Omega \subset \R^{n+1}$ be a $2$-sided chord-arc domain. There exist constants $C \ge 1$ and $\Lambda \ge 1$ such that for every $\Delta = \Delta(y,r)$ we have
  \begin{align}
    \label{inequality:1-poincare} \fint_{\Delta} |f(x) - \langle f \rangle_\Delta| \, d\sigma(x) \le C r \fint_{\Lambda \Delta} \rho(x) \, d\sigma(x),
  \end{align}
  for any $f \in L^1_{\loc}(\pom)$ and any upper gradient $\rho$ of $f$ (see Definition \ref{defin:upper_gradient}), where $\sigma \coloneqq \Hc^n|_{\pom}$ is the surface measure.
\end{corollary}
We note that this weak $1$-Poincar\'e inequality implies a weak $p$-Poincar\'e inequality (that is, inequality \eqref{inequality:1-poincare} with the right hand side replaced by $C r (\fint_{\Lambda \Delta} |\rho(x)|^p \, d\sigma(x))^{1/p}$) for any $1 < p < \infty$ by H\"older's inequality. Furthermore, by \cite[Corollary 9.14]{heinonenetal}, these weak $p$-Poincar\'e inequalities imply weak $(p,p)$-Poincar\'e inequalities of the type
\begin{align*}
  \left( \fint_{\Delta} |f(x) - \langle f \rangle_\Delta|^p \, d\sigma(x) \right)^{1/p} \le \widetilde{C} r \left( \fint_{\Lambda \Delta} |\rho(x)|^p \, d\sigma(x) \right)^{1/p},
\end{align*}
where we used the same notation as in Corollary \ref{corollary:poincare}. See also Corollary \ref{corollary:poincare_for_hmt_gradient} for an inequality of this type for the Hofmann--Mitrea--Taylor Sobolev spaces in bounded $2$-sided chord-arc domains. We also note that the conclusion of Corollary \ref{corollary:poincare} is the inequality appearing in Theorem \ref{theorem:mourgogloutolsa}. It is natural to ask if Corollary \ref{corollary:poincare} can be strengthened into a characterization but this is not possible: there exist non-chord-arc domain open sets $\Omega \subset \R^{n+1}$ with $n$-dimensional Ahlfors--David regular boundaries $\pom$ that support weak Poincar\'e inequalities. See Section \ref{section:questions} for an example and discussion. However, we point out that a result to the converse direction was recently proven by Azzam who showed that weak Poincar\'e inequalities imply uniform rectifiability (see Definition \ref{defin:ur}) for Ahlfors--David regular sets $E \subset \R^{n+1}$ \cite{azzam_poincare}.

Furthermore, Corollary \ref{corollary:poincare} combined with the work of Heinonen and Koskela \cite[Theorem 5.7]{heinonenkoskela}, Korte \cite[Theorem 3.3]{korte}, and Cheeger \cite[Theorem 17.1]{cheeger} (see also e.g. \cite{merhej} and \cite{heinonenetal}) immediately give us the following two new results about the geometric structure of boundaries of $2$-sided chord-arc domains:

\begin{corollary}
  Let $\Omega \subset \R^{n+1}$ be a $2$-sided chord-arc domain. Then $\pom$ is a Loewner space (see Definition \ref{defin:lowener}).
\end{corollary}

Examples of Loewner spaces include the Euclidean space, Carnot groups and Riemannian manifolds of nonnegative Ricci curvature \cite[Section 6]{heinonenkoskela}. For other examples, see \cite[Section 6]{heinonenkoskela} and \cite[Section 14.2]{heinonenetal}.

\begin{corollary}
  \label{corollary:annular_quasiconvexity}
  Let $\Omega \subset \R^{n+1}$, be a $2$-sided chord-arc domain. Then
  \begin{enumerate}
    \item[$\bullet$] if $n=1$, then $\pom$ is quasiconvex (see Definition \ref{defin:quantitative_connectivity}),
    \item[$\bullet$] if $n>1$, then $\pom$ is annularly quasiconvex (see Definition \ref{defin:quantitative_connectivity}).
  \end{enumerate}
\end{corollary}

We note that the case $n=1$ in Corollary \ref{corollary:annular_quasiconvexity} cannot be improved to annular quasiconvexity: $\Omega = B(0,1)$ (the unit disc) is a $2$-sided chord-arc domain but for any $z \in \pom$ and any $0 < r < \tfrac{1}{2}$ there exist points $x,y \in B(z,2r) \setminus B(z,r)$ such that $x$ and $y$ can be joined in $\pom \setminus \{z\}$ only with paths $\gamma$ such that $\ell(\gamma) \ge 1$.

The proofs of Theorem \ref{theorem:local_john_harnack}, Theorem \ref{theorem:boundary_poincare} and Corollary \ref{corollary:poincare} utilize significant advances in geometric analysis over the past 25 years. For Theorem \ref{theorem:local_john_harnack}, we use harmonic measure theory (particularly the very recent results of Azzam, Hofmann, Martell, Mourgoglou and the second named author \cite{azzametal1}) and uniform rectifiability techniques (particularly the Bilateral Weak Geometric Lemma of David and Semmes \cite{davidsemmes_analysis}). For Theorem \ref{theorem:boundary_poincare} and Corollary \ref{corollary:poincare}, we combine layer potential techniques of Hofmann, Mitrea and Taylor \cite{hofmannmitreataylor} and pointwise and Lipschitz characterizations of Poincar\'e inequalities of Heinonen and Koskela \cite{heinonenkoskela} and Keith \cite{keith} (see also \cite{heinonen}) with suitable localization and truncation arguments. One of the novelties in the proof of Theorem \ref{theorem:boundary_poincare} is the use of a weak type $(1,1)$ version of a weak Poincar\'e inequality with a tail, analogous to the strong type $(p,p)$ version proved previously in \cite{hofmannmitreataylor}.

The paper is organized as follows. In Section \ref{section:notation} we fix the basic notation, review the numerous definitions needed in the paper and consider some auxiliary results from the literature. In Section \ref{section:john_conditions}, we define three different John type conditions and compare them. In Section \ref{section:bwgl}, we consider the Bilateral Weak Geometric Lemma of David and Semmes and prove some straightforward related results for the proof of Theorem \ref{theorem:local_john_harnack}, and in Section \ref{section:local_john_harnack} we prove Theorem \ref{theorem:local_john_harnack}. In Section \ref{section:hmt_and_quasiconvex}, we consider the Hofmann--Mitrea--Taylor type weak $p$-Poincar\'e inequality with a tail for $1 < p < \infty$ and use techniques from its proof to prove a weak type estimate for the case $p=1$. In Section \ref{section:poincare}, we use this weak type estimate together with some key results from the theory of Poincar\'e inequalities in metric spaces to prove Theorem \ref{theorem:boundary_poincare} and Corollary \ref{corollary:poincare}. In Section \ref{section:questions}, we end the paper by constructing an example that shows us that the assumptions in Corollary \ref{corollary:poincare} are not optimal and consider some questions related to this work.

\subsection*{Acknowledgments}

The authors would like to thank Steve Hofmann for kindly answering their questions related to the work in \cite{hofmannmitreataylor} and for his questions and comments that ended up improving and sharpening the presentation of this paper. They would also like to thank Sylvester Eriksson-Bique and Pekka Koskela for their helpful comments related open problems in the field of Poincar\'e inequalities. The second named author would also like to thank Mihalis Mourgoglou for helpful conversations particularly related to Example \ref{example:non-2-sided-cad_poincare}. Part of this work was completed during the ``Interactions between Geometric measure theory, Singular integrals, and PDE'' Trimester Program at the Hausdorff Research Institute for Mathematics (HIM) in the spring 2022. The authors would like to thank the organizers of the program and the people at HIM for the kind hospitality shown during their stay.

\section{Notation, basic definitions and auxiliary results}
\label{section:notation}

\begin{defin}
  \label{defin:2-sided_conditions}
  Let $\Omega \subset \R^{n+1}$ be an open set and consider a geometric condition (such as the local John condition (see Definition \ref{defin:john_conditions})). If the interior of $\Omega^{\text{c}}$ satisfies the condition, we say that $\Omega$ satisfies the \emph{exterior} version of the condition. If both $\Omega$ and the interior of $\Omega^{\text{c}}$ satisfy the same condition, we say that $\Omega$ satisfies the \emph{$2$-sided} version of the same condition.
\end{defin}

We use the following basic notation and terminology:
\begin{enumerate}
  \item[$\bullet$] $\Omega \subset \R^{n+1}$ is an open set with $n$-dimensional boundary $\pom$. We denote the surface measure of $\pom$ by $\sigma \coloneqq \Hc^n|_{\pom}$. Unless explicitly mentioned, we assume that $\pom$ is Ahlfors--David regular (see Definition \ref{defin:adr}) and that both $\Omega$ and $\text{int} \, \Omega^{\text{c}}$ satisfies the corkscrew condition (that is, $\Omega$ satisfies the $2$-sided corkscrew condition; see Definition \ref{defin:2-sided_conditions} and Definition \ref{defin:corkscrew}).
  
  \item[$\bullet$] Usually, we use capital letters $X,Y,Z$, and so on to denote points in $\Omega$, and lowercase letters $x,y,z$, and so on to denote points in $\pom$.
  
  \item[$\bullet$] For every point $X \in \R^{n+1}$, we deonte $\delta(X) \coloneqq \dist(X,\pom)$.
  
  \item[$\bullet$] We denote the open $(n+1)$-dimensional Euclidean ball with radius $r > 0$ by $B(X,r)$ or $B(x,r)$, depending on whether the center point lies in $\Omega$ or $\pom$. For any $x \in \pom$ and any $r > 0$, we denote the surface ball centered at $x$ with radius $r$ by $\Delta(x,r) \coloneqq B(x,r) \cap \pom$.
  
  \item[$\bullet$] Given an Euclidean ball $B \coloneqq B(X,r)$ or a surface ball $\Delta \coloneqq \Delta(x,r)$ and a constant $\kappa > 0$, we denote $\kappa B \coloneqq B(X,\kappa r)$ and $\kappa \Delta \coloneqq \Delta(x,\kappa r)$.
  
  \item[$\bullet$] For a metric measure space $(X,d,\mu)$, a function $f$ and an open ball $B$, we denote the average of $f$ over $B$ by
  \begin{align*}
    \langle f \rangle_B \coloneqq \fint_B f \, d\mu \coloneqq \frac{1}{\mu(B)} \int_B f \, d\mu.
  \end{align*}

  \item[$\bullet$] A \emph{path} is a continuous function $\gamma \colon [0,1] \to X$, where $X$ is a metric space. With slight abuse of terminology, we call a path $\gamma \colon [0,1] \to \overline{\Omega}$ a \emph{path in $\Omega$} if $\gamma(t) \in \Omega$ for every $t \in (0,1]$. With slight abuse of notation, we denote $Z \in \gamma$ if there exists $t \in [0,1]$ such that $\gamma(t) = Z$. We say that a path $\gamma$ is \emph{from $X_1$ to $X_2$} if $\gamma(0) = X_1$ and $\gamma(1) = X_2$.
  
  \item[$\bullet$] The \emph{length of a path} $\gamma \colon [0,1] \to \overline{\Omega}$ is defined as
  \begin{align*}
    \ell(\gamma) \coloneqq \sup \left\{ \sum_{i=0}^k \left|\gamma(t_i) - \gamma(t_{i+1})\right| \right\},
  \end{align*}
  where the supremum is taken over all finite partitions $0 = t_0 < t_1 < \cdots < t_k = 1$ of the interval $[0,1]$. We say that a path $\gamma$ is \emph{rectifiable} if the length of $\gamma$ is finite.
  
  \item[$\bullet$] Given a rectifiable path $\gamma$ and a function $f$, we denote the \emph{arc-length parametrization of $\gamma$} (that is, the re-parametrization of $\gamma$ with respect to $\ell(\gamma)$) by $\gamma_\ell \colon [0,\ell(\gamma)] \to \R^{n+1}$ and the integral of $f$ over $\gamma$ by
  \begin{align*}
    \int_\gamma f \coloneqq \int_0^{\ell(\gamma)} f \circ \gamma_\ell(t) \, dt.
  \end{align*}

  \item[$\bullet$] For a path $\gamma$ and points $X_1, X_2 \in \gamma$ with $\gamma(t) = X_1$ and $\gamma(s) = X_2$ for $t,s \in [0,1]$, $t < s$, we denote the piece of $\gamma$ from $X_1$ to $X_2$ by $\gamma(X_1,X_2)$ and its length by $\ell(\gamma(X_1,X_2))$. Again, with slight abuse of notation, we denote $Z \in \gamma(X_1,X_2)$ if there exists $u \in (s,t)$ such that $\gamma(u) = Z$.
  
  \item[$\bullet$] We denote harmonic measure with pole at $X \in \Omega$ by $\omega^X$. Usually, we drop the pole from the notation if we consider properties that hold for every $X \in \Omega$.
  
  \item[$\bullet$] Let $A \subset \R^{n+1}$, $f \colon A \to \R$, $\alpha > 0$ and $\beta \ge 1$. We say that $f$ is \emph{$\alpha$-Lipschitz} if
  \begin{align*}
    |f(x) - f(y)| \le \alpha|x-y|
  \end{align*}
  for all $x,y \in A$. We say that $f$ is \emph{locally $\alpha$-Lipschitz} if
  \begin{align}
    \label{defin:locally_lipschitz} \limsup_{\substack{A \ni y \to x \\ y \neq x}} \frac{|f(x) - f(y)|}{|x-y|} \le \alpha
  \end{align}
  for every $x \in A$. We say that $f$ is \emph{$\beta$-bi-Lipschitz} if
  \begin{align*}
    \frac{1}{\beta} |x-y| \le |f(x) - f(y)| \le \beta|x-y|
  \end{align*}
  for all $x,y \in A$.
  
  \item[$\bullet$] We denote the \emph{measure-theoretic boundary of $\Omega$} by $\partial_* \Omega$: we have $x \in \partial_* \Omega$ if and only if $x \in \pom$, and
  \begin{align}
    \label{defin:measure-theoretic_boundary}
    \liminf_{r \to 0^+} \frac{|B(x,r) \cap \Omega|}{r^{n+1}} > 0
    \ \ \ \text{ and } \ \ \ 
    \liminf_{r \to 0^+} \frac{|B(x,r) \setminus \overline{\Omega}|}{r^{n+1}} > 0.
  \end{align}
  
  \item[$\bullet$] For any $p > 1$, we denote the \emph{H\"older conjugate of $p$} by $p'$. The numbers $p$ and $p'$ satisfy $1/p + 1/p' = 1$.
  
  \item[$\bullet$] We denote the \emph{non-tangential maximal operator} by $N_*$: for a function $u$ in $\Omega$, $N_* u$ is a function $\pom$ defined as
  \begin{align}
    \label{defin:non-tangential_max} N_* u(x) \coloneqq \sup_{Y \in \Gamma_\alpha(x)} |u(Y)|,
  \end{align}
  where $\Gamma_\alpha(x)$ is the \emph{cone at $x \in \pom$ with aperture $\alpha$},
  \begin{align}
    \label{defin:cone} \Gamma_\alpha(x) \coloneqq \{Y \in \R^{n+1} \colon |x-Y| < \alpha \, \delta(Y)\}.
  \end{align}
  We say that a function $u$ in $\Omega$ \emph{converges non-tangentially} to a function $f$ on $\pom$ if $u(Y) \to f(x)$ as $Y \to x$ inside $\Gamma_\alpha(x)$.
  
  \item[$\bullet$] The letters $c$ and $C$ and their ovious variations denote constants that depend only on dimension, ADR constant (see Definition \ref{defin:adr}), UR constants (see Definition \ref{defin:ur}) and other similar parameters. The values of $c$ and $C$ may change from one occurence to another. We do not track how our bounds depend on these constants and usually just write $\alpha_1 \lesssim \alpha_2$ if $\alpha_1 \le c \, \alpha_2$ for a constant like this $c$ and $\alpha_1 \approx \alpha_2$ if $\alpha_1 \lesssim \alpha_2 \lesssim \alpha_1$. If the constant $c_\kappa$ depends only on parameters of the previous type and some other parameter $\kappa$, we usually write $\alpha_1 \lesssim_\kappa \alpha_2$ instead of $\alpha_1 \le c_\kappa \alpha_2$.
\end{enumerate}

\subsection{ADR, UR, NTA, CAD, and corkscrew condition}

\begin{defin}[ADR]
  \label{defin:adr}
  We say that a closed set $E \subset \R^{n+1}$ is a \emph{$d$-ADR (Ahlfors--David regular)} set if there exists a constant $D \ge 1$ such that
  \begin{align*}
    \frac{1}{D} r^d \le \Hc^d(B(x,r) \cap E) \le D r^d
  \end{align*}
  for every $x \in E$ and every $r \in (0,\diam(E))$, where $\diam(E)$ may be infinite.
\end{defin}

\begin{defin}[Corkscrew condition]
  \label{defin:corkscrew}
  We say that $\Omega$ satisfies the \emph{corkscrew condition} if there exists a constant $c \in (0,1)$ such that for every surface ball $\Delta \coloneqq \Delta(x,r)$ with $x \in \partial\Omega$ and $0 < r < \diam(\partial \Omega)$ there exists a point $X_\Delta \in \Omega$ such that $B(X_\Delta, cr) \subset B(x,r) \cap \Omega$,
\end{defin}

\begin{defin}[UR]
  \label{defin:ur}
  Following \cite{davidsemmes_singular, davidsemmes_analysis}, we say that an $n$-ADR set $E \subset \R^{n+1}$ is \emph{UR (uniformly rectifiable)} if it contains \emph{big pieces of Lipschitz images} of $\R^n$, i.e., there exist constants $\theta, M > 0$ such that for every $x \in E$ and $r \in (0,\diam(E))$ there is a Lipschitz mapping
  $\rho = \rho_{x,r} \colon \R^n \to \R^{n+1}$, with Lipschitz norm no larger that $M$, such that
  \begin{align*}
    \Hc^n(E \cap B(x,r) \cap \rho(\{y \in \R^n \colon |y| < r\})) \ge \theta r^n.
  \end{align*}
\end{defin}

\begin{defin}[NTA]
  \label{defin:nta}
  Following \cite{jerisonkenig}, we say that a domain $\Theta \subset \R^{n+1}$ is \emph{NTA (non-tangentially accessible)} if
  \begin{enumerate}
    \item[$\bullet$] $\Theta$ satisfies the \emph{Harnack chain condition}: there exists a 
                     uniform constant $C$ 
                     such that for every $\rho > 0$, $\Lambda \ge 1$ and $X,X' \in \Theta$ with $\delta(X), \delta(X') \ge \rho$
                     and $|X - X'| < \Lambda \rho$ there exists a chain of open balls $B_1, \ldots, B_N \subset \Theta$, $N \le C(\Lambda)$,
                     with $X \in B_1$, $X' \in B_N$, $B_k \cap B_{k+1} \neq \emptyset$ and $C^{-1} \diam(B_k) \le \dist(B_k,\partial \Theta) \le C \diam(B_k)$,
    
    \item[$\bullet$] both $\Theta$ and $\R^{n+1} \setminus \Theta$ satisfy the corkscrew condition.
  \end{enumerate}
\end{defin}

\begin{defin}[CAD] \label{CAD}
  \label{defin:cad}
  An open set $\Omega\subset \ree$ is a {\em CAD (chord-arc domain)} if it is NTA, and $\pom$ is $n$-ADR.
\end{defin}

The following result originates from the works of David and Jerison \cite{davidjerison} and Semmes \cite{semmes_analysis} (see also \cite[Definition 3.7 and Corollary 3.9]{hofmannmitreataylor}):
\begin{theorem}
  \label{theorem:corkscrew_ur}
  Suppose that $\Omega \subset \R^{n+1}$ is an open set satisfying the two-sided corkscrew condition and that $\pom$ is ADR. Then $\pom$ is UR and $\sigma(\pom \setminus \partial_* \Omega) = 0$.
\end{theorem}

\subsubsection{Dyadic cubes}

\begin{theorem}[E.g. {\cite{christ, sawyerwheeden, hytonenkairema}}]
  \label{theorem:existence_of_dyadic_cubes}
  Suppose that $E$ is a $d$-ADR set. Then there exists a countable collection $\D$ (that we call a \emph{dyadic system}),
  \begin{align*}
    \D \coloneqq \bigcup_{k \in \Z} \D_k, \ \ \ \ \ \D_k \coloneqq \{ Q_\alpha^k \colon \alpha \in \mathcal{A}_k \}
  \end{align*}
  of Borel sets $Q_\alpha^k$ (that we call \emph{(dyadic) cubes}) such that
  \begin{enumerate}
    \item[(i)] the collection $\D$ is \emph{nested}: $\text{if } Q,P \in \D, \text{ then } Q \cap P \in \{\emptyset,Q,P\}$,
    \item[(ii)] $E = \bigcup_{Q \in \D_k} Q$ for every $k \in \Z$ and the union is disjoint,
    \item[(iii)] there exist constants $c_1 > 0$ and $C_1 \ge 1$ such that
                     \begin{align}
                       \label{dyadic_cubes:balls_inclusion} \Delta(z_\alpha^k,c_1 2^{-k}) \subseteq Q_\alpha^k \subseteq \Delta(z_\alpha^k, C_1 2^{-k}),
                     \end{align}
    \item[(iv)] for every set $Q_\alpha^k$ there exists at most $N$ cubes $Q_{\beta_i}^{k+1}$ (called the \emph{children} of $Q_\alpha^k$) such that $Q_\alpha^k = \bigcup_i Q_{\beta_i}^{k+1}$, 
               where the constant $N$ depends only on the ADR constant of $E$.
  \end{enumerate}
\end{theorem}

\begin{notation}
  \label{notation:dyadic_cubes} We shall use the following notational conventions.
  \begin{enumerate}
    \item[(1)] For each $k$, and for every cube $Q_\alpha^k \coloneqq Q \in \D_k$, we denote $\ell(Q) \coloneqq C_1 2^{-k}$ and $x_Q \coloneqq z_\alpha^k$. We call $\ell(Q)$ the \emph{side length} of $Q$, and $x_Q$ the \emph{center} of $Q$. If the set $E$ is bounded or disconnected, the side length might not be well-defined, but we can fix this problem easily by, for example, considering the minimum of the numbers $C_1 2^{-k}$ such that $Q \subset \Delta(x_Q,C_1 2^{-k})$.
    
    \item[(2)] For every $Q = Q_\alpha^k$ and $\kappa \ge 1$, we denote
    \begin{align*}
      \kappa B_Q \coloneqq B(z_\alpha^k, \kappa \ell(Q)).
    \end{align*}
    For $\kappa = 1$, we simply denote $\kappa B_Q = B_Q$.
  \end{enumerate}
\end{notation}

\begin{defin}
  \label{defin:carleson_packing_norm}
  We say that a collection $\Ac \subset \D$ satisfies a \emph{Carleson packing condition} if there exists a constant $C \ge 1$ such that
  \begin{align*}
    \sum_{Q \in \Ac, Q \subset Q_0} \sigma(Q) \le C \sigma(Q_0)
  \end{align*}
  for every cube $Q_0 \in \D$. We call the smallest such constant $C$ the \emph{Carleson packing norm of $\Ac$} and denote it by $\Cs_{\Ac}$.
\end{defin}

We need the following straightforward lemma in the proof of Theorem \ref{theorem:local_john_harnack}:
\begin{lemma}
  \label{lemma:subset_carleson}
  Let $\Ac \subset \D$ be a collection satisfying a Carleson packing condition. Also, let $Q_0 \in \D$ be a cube and $A \subset Q_0$ a measurable subset such that $\sigma(A) \ge c \, \sigma(Q_0)$ for a constant $c \in (0,1)$. Then there exists a cube $Q \in \D \setminus \Ac$ such that $\sigma(Q \cap A) > 0$ and $\ell(Q) \approx_{\Cs_{\Ac},c} \ell(Q_0)$.
\end{lemma}

\begin{proof}
  Let us consider the first $K > \Cs_{\Ac} / c$ generations of subcubes of $Q_0$. Each of these generations covers the set $A$. For contradiction, suppose that $\sigma(Q \cap A) = 0$ for each of these subcubes such that $Q \notin \Ac$. Then, for every $m=1,\ldots,K$, we can cover the set $A$ (up to a set of measure $0$) by cubes from the collection $\{Q \in \Ac \colon \ell(Q) = 2^{-m}\ell(Q_0)\}$. In particular, we get
  \begin{align*}
    \sum_{Q \in \Ac, Q \subset Q_0} \sigma(Q)
    = \sum_{m=0}^\infty \sum_{\substack{Q \in \Ac \\ \ell(Q) = 2^{-m} \ell(Q_0)}} \sigma(Q)
    &\ge \sum_{m=1}^K \sum_{\substack{Q \in \Ac \\ \ell(Q) = 2^{-m} \ell(Q_0)}} \sigma(Q \cap A) \\
    &= \sum_{m=1}^K \sigma(A)
    \ge K c \, \sigma(Q_0)
    > \Cs_{\Ac} \, \sigma(Q_0),
  \end{align*}
  which contradicts the Carleson packing condition. Hence, there exists at least one cube $Q$ from the the first $\lceil \Cs_{\Ac} / c \rceil$ generations of subcubes of $Q_0$ such that $\sigma(Q \cap A) > 0$ and $Q \in \D \setminus \Ac$.
\end{proof}

\subsection{Harmonic measure and the weak-$A_\infty$ condition}

\begin{defin}[Weak $A_\infty$ for harmonic measure]
  For harmonic measure $\omega$, we denote $\omega \in \text{weak-}A_\infty(\sigma)$ if there exist constants $C \ge 1$ and $s > 0$ such that if $B \coloneqq B(x,r)$ with $x \in \pom$ and $r \in (0,\diam(\pom)/4)$ and $A \subset \Delta \coloneqq B \cap \pom$ is a Borel set, then
  \begin{align}
    \label{condition:weak_a_infty} \omega^Y(A) \le C \left( \frac{\sigma(A)}{\sigma(\Delta)} \right)^s \omega^Y(2\Delta)
  \end{align}
  for every $Y \in \Omega \setminus 4B$.
\end{defin}

We note that the constant $2$ in \eqref{condition:weak_a_infty} can be replaced with any constant $c > 1$ without changing the class weak-$A_\infty(\sigma)$ and that the weak-$A_\infty$ property is equivalent with a weak Reverse H\"older property for the Radon--Nikodym derivative (see e.g. \cite[Section 8]{andersonhytonentapiola}).

We use the following lemma from \cite{azzametal1} in the proof of Theorem \ref{theorem:local_john_harnack}. The lemma is a key ingredient for the proof of the geometric characterization of the $\text{weak-}A_\infty$ property of harmonic measure:

\begin{lemma}[{\cite[Section 10]{azzametal1}}]
  \label{lemma:connective_set}
  Suppose that $\Omega$ has a uniformly rectifiable boundary $\pom$ and that $\omega \in \text{weak-}A_\infty$. Suppose also that $R_0 \in \D(\pom)$ is a dyadic cube and $Y \in \Omega \setminus 4B_{R_0}$ is a point such that
  \begin{align*}
    c_1 \ell(R_0) \le \delta(Y) \le \dist(Y,R_0) \le c_1^{-1} \ell(R_0)
  \end{align*}
  and $\omega^Y(R_0) \ge c_2 > 0$. Then there exist a constant $c_3 > 0$ and a subset $A \subset R_0$ such that $\sigma(A) \ge c_3 \, \sigma(R_0)$ and each point $x \in A$ can be joined to $Y$ by a $D$-non-tangential path (see Definition \ref{defin:nt-paths}), where $c_3$ and $D$ depend only on $c_1$, $c_2$, $n$, the weak-$A_\infty$ constants and the uniform rectifiability constants.
\end{lemma}

We also need the following classical estimate (sometimes referred to as \emph{Bourgain's estimate} \cite[Lemma 1]{bourgain}):
\begin{lemma}
  \label{lemma:bourgain}
  There exist uniform constants $c_0 \in (0,1)$ and $C_0 > 1$, depending only on $n$ and the ADR constant, such that the following holds: if $x \in \pom$, $r \in (0,\diam(\pom))$ and $Y \in B(x,c_0r)$, then $\omega^Y(\Delta(x,r)) \ge 1/C_0 > 0$.
\end{lemma}

\subsection{Quasiconvexity, annular quasiconvexity and Loewner spaces}

\begin{defin}
  Let $(X,d)$ be a metric space. We say that a nonempty set $E \subset X$ is a \emph{continuum} if it is compact and connected. We call a continuum \emph{non-degenerate} if it contains more than one point. We say that points $x,y \in F \subset X$ can be \emph{joined in $F$} if there exists a continuum $E \subset F$ such that $x,y \in E$.
\end{defin}

\begin{defin}
  \label{defin:quantitative_connectivity}
  Let $(X,d)$ be a metric space. We say that $X$ is
  \begin{enumerate}
   \item[i)] \emph{quasiconvex} if there exists a constant $C \ge 1$ such that for any pair of points $x,y \in X$ there exists a path $\gamma_{x,y} \colon [0,1] \to X$ such that $\gamma_{x,y}(0) = x$, $\gamma_{x,y}(1) = y$ and $\ell(\gamma_{x,y}) \le C \, d(x,y)$,
   
   \item[ii)] \emph{annularly quasiconvex} if there exists a constant $C \ge 1$ such that if $x,y \in B(z,2r) \setminus B(z,r)$ for $z \in X$ and $r < \tfrac{1}{C} \diam(X)$, then $x$ and $y$ can be joined by a path $\gamma = \gamma_{x,y}$ in $B(z,Cr) \setminus B(z,r/C)$ such that $\ell(\gamma) \le C \, d(x,y)$.
    
    \item[iii)] \emph{rectifiably connected} if for any pair of points $x,y \in X$ there exists a rectifiable path $\gamma_{x,y}$ from $x$ to $y$.
  \end{enumerate}
\end{defin}

\begin{defin}
  Let $(X,d,\mu)$ be a rectifiably connected metric measure space with a locally finite Borel measure $\mu$. Let $E,F \subset X$ be two disjoint nondegenerate continua, $(E,F) = (E,F;X)$ be the family of paths in $X$ connecting $E$ and $F$ and $p \ge 1$. We define the \emph{$p$-modulus of $(E,F)$} as
  \begin{align*}
    \text{mod}_p (E,F) \coloneqq \inf_\varrho \int_X \varrho^p \, d\mu,
  \end{align*}
  where the infimum is taken over all nonnegative Borel functions $\varrho \colon X \to [0,\infty)$ satisfying
  \begin{align*}
    \int_\gamma \varrho \, ds \ge 1
  \end{align*}
  for every $\gamma$ in $(E,F)$.
\end{defin}

\begin{defin}
  \label{defin:lowener}
  Let $(X,d,\mu)$ be a $d$-dimensional, rectifiably connected metric measure space with a locally finite Borel measure $\mu$. We say that $X$ is a \emph{$d$-Loewner space} if there exists a function $\phi \colon (0,\infty) \to (0,\infty)$ such that if $E$ and $F$ are two disjoint, nondegenerate continua in $X$ satisfying
  \begin{align*}
    \frac{\dist(E,F)}{\min\{\diam(E),\diam(F)\}} \le t,
  \end{align*}
  then $\phi(t) \le \text{mod}_d (E,F)$.
\end{defin}

\subsection{Upper, Haj{\l}asz, tangential and Hofmann--Mitrea--Taylor gradients, Sobolev spaces, and weak Poincar\'e inequalities on $\pom$}
\label{subsection:gradients}

\begin{defin}
  \label{defin:upper_gradient}
  Let $f \colon \pom \to \R$. We say that a Borel-measurable function $\rho \colon \pom \to [0,\infty]$ is an \emph{upper gradient of $f$} if we have
  \begin{align*}
    |f(x) - f(y)| \le \int_\gamma \rho = \int_0^{\ell(\gamma)} \rho \circ \gamma_\ell(t) \, dt
  \end{align*}
  for all $x,y \in \pom$ and every rectifiable path $\gamma$ from $x$ to $y$, where $\gamma_\ell \colon [0,\ell(\gamma)] \to \R$ is the arc-length parametrization of $\gamma$.
\end{defin}
The notion of upper gradients originates from \cite{heinonenkoskela_essay, heinonenkoskela} where they were called \emph{very weak gradients}.

\begin{defin}
  \label{defin:hajlasz}
  Let $f \colon \pom \to \R$. We say that a Borel-measurable function $g \colon \pom \to \R$ is a \emph{Haj{\l}asz gradient of $f$} if we have
  \begin{align*}
    |f(x) - f(y)| \le (g(x) + g(y))|x-y|
  \end{align*}
  for almost every $x,y \in \pom$. We denote the class of all Haj{\l}asz gradients of $f$ by $D(f)$. For $p \ge 1$, we denote the space of Borel functions with a Haj{\l}asz gradient in $L^p(\pom)$ by $\dot{W}^{1,p}(\pom)$ and define a seminorm $\|\cdot\|_{\dot{W}^{1,p}(\pom)}$ for this space by setting
  \begin{align*}
    \|f\|_{\dot{W}^{1,p}(\pom)} \coloneqq \inf_{g \in D(f)} \|g\|_{L^p(\pom)}.
  \end{align*}
\end{defin}
The notion of Haj{\l}asz gradients originates from \cite{hajlasz}. By \cite{jiangetal}, if $f, g \in L^1_{\loc}(\pom)$ and $g$ is a Haj{\l}asz gradient of $f$, then there exist functions $\widetilde{f}, \widetilde{g} \in L^1_{\loc}(\pom)$ such that $f = \widetilde{f}$ and $g = \widetilde{g}$ almost everywhere and $4\widetilde{g}$ is an upper gradient of $\widetilde{f}$.

\begin{defin}
  \label{defin:tangential:gradient}
  Let $f \colon \pom \to \R$ be a Lipschitz function and let $x \in \pom$ be a point such that the approximate tangent plane $T_x \pom$ exists. Let $\widetilde{f} \colon \R^{n+1} \to \R$ be any Lipschitz extension of $f$ to $\R^{n+1}$. We say that $f$ is \emph{tangentially differentiable at $x$} if $\widetilde{f}|_{x + T_x \pom}$ is differentiable at $x$. When it exists, we denote the corresponding \emph{tangential gradient at $x$} by $\nablat f(x)$.
\end{defin}
A thorough reference for these types of differentiability results is the book of Maggi \cite{maggi}. In particular, since $\pom$ is uniformly rectifiable by Theorem \ref{theorem:corkscrew_ur} (and therefore $\pom$ is rectifiable; see \cite[p.\,2629]{hofmannmitreataylor} for an explicit proof), we know that the approximate tangent plane exists for almost every point $x \in \pom$ by \cite[Theorem 10.2]{maggi}, the tangential gradient exists for $\sigma$-a.e. point $x \in \pom$ by \cite[Theorem 11.4]{maggi} and the definition of the tangential gradient is independent of the choice of the Lipschitz extension by \cite[Theorem 10.1, Proposition 10.5, and Lemma 11.5]{maggi}.

By \cite{hofmannmitreataylor}, the next definitions make sense if we know that -- on top of the standing assumptions on $\Omega$ and $\pom$ (see the beginning of Section \ref{section:notation}) -- $\Omega$ is a set of locally finite perimeter and $\sigma(\pom \setminus \partial_* \Omega) = 0$ (recall \eqref{defin:measure-theoretic_boundary}). Since we only use the following objects when $\Omega$ satisfies the $2$-sided local John condition (or, equivalently by Corollary \ref{corollary:local_john_cad}, when $\Omega$ is a $2$-sided chord-arc domain), the assumptions are automatically satisfied by Theorem \ref{theorem:corkscrew_ur} and \cite[Corollary 3.14]{hofmannmitreataylor}.
\begin{defin}
  \label{defin:hmt_gradient}
  Let $\nu(x) \coloneqq (\nu_j(x))_{j=1}^{n+1}$ be the outer unit normal at $x \in \pom$. For a function $\varphi \colon \R^{n+1} \to \R$, $\varphi \in C^1_{\text{c}}$, we define the \emph{(Hofmann--Mitrea--Taylor) tangential derivatives of $\varphi$} as
  \begin{align*}
    \partial_{\smallt,j,k} \varphi \coloneqq \nu_j (\partial_k \varphi)|_{\pom} - \nu_k (\partial_j \varphi)|_{\pom}
  \end{align*}
  for $1 \le j,k \le n+1$. The \emph{Hofmann--Mitrea--Taylor Sobolev space} $L_1^p(\pom)$ is the space of the functions $f \in L^p(\pom)$ such that there exists a finite constant $C_f$ such that
  \begin{align*}
    \sum_{1 \le j,k \le n+1} \left| \int_{\pom} f \partial_{\smallt,k,j} \varphi \, d\sigma \right| \le C_f \|\varphi\|_{L^{p'}(\sigma)}
  \end{align*}
  for every $\varphi \in C^\infty_{\text{c}}(\R^{n+1})$. By the Riesz representation theorem, for every $f \in L^p_1(\pom)$ and each $j,k = 1, 2, \ldots, n+1$ there exists a function $h_{j,k} \in L^p(\pom)$ satisfying
  \begin{align*}
    \int_{\pom} h_{j,k} \varphi \, d\sigma = \int_{\pom} f \partial_{\smallt,k,j} \varphi \, d\sigma
  \end{align*}
  for every $\varphi \in C^\infty_{\text{c}}(\R^{n+1})$. We set $\partial_{\smallt,j,k} f \coloneqq h_{j,k}$ and define the \emph{Hofmann--Mitrea--Taylor gradient} $\nablahmt f$ by setting
  \begin{align*}
    \nablahmt f \coloneqq \left( \sum_k \nu_k \partial_{\smallt,j,k} f \right)_{j=1}^{n+1}.
  \end{align*}
\end{defin}
For the comprehensive theory of Hofmann--Mitrea--Taylor Sobolev spaces, see particularly Section 3 and 4 in \cite{hofmannmitreataylor}.

\begin{defin}
  \label{defin:weak_poincare} Let $1 \le p < \infty$.
  We say that $\pom$ supports a \emph{weak (Heinonen--Koskela type) $p$-Poincar\'e inequality} if there exist constants $C = C_p \ge 1$ and $\Lambda \ge 1$ such that for every $\Delta = \Delta(y,r)$ we have
  \begin{align*}
    \fint_{\Delta} |f(x) - \langle f \rangle_\Delta| \, d\sigma(x) \le C r \left(\fint_{\Lambda \Delta} \rho(x)^p \, d\sigma(x)\right)^{1/p},
  \end{align*}
  for any $f \in L^1_{\loc}(\pom)$ and any upper gradient $\rho$ of $f$.
\end{defin}

\subsection{Solvability for the Laplacian}

\begin{defin}
  \label{defin:lp-sovability}
  Let $1 \le p < \infty$.
  We say that the Dirichlet problem (for the Laplacian) is \emph{solvable in $L^p$ for $\Omega$} if there exists a constant $C$ such that for any continuous function $f \in C(\pom)$ the solution $u = u_f$ to the Dirichlet problem with datum $f$ converges non-tangentially to $f$ $\sigma$-a.e. and
  \begin{align*}
    \|N_* u\|_{L^p(\partial \Omega)} \le C \|f\|_{L^p(\partial \Omega)},
  \end{align*}
  where $N_*$ is the non-tangential maximal operator (recall \eqref{defin:non-tangential_max}).
\end{defin}

\begin{defin}
  \label{defin:regularity_problem}
  Let $1 \le p < \infty$.
  We say that the \emph{regularity problem (for the Laplacian) is solvable in $L^p$ for $\Omega$} if there exists a constant $C$ such that for any Lipschitz function $f \colon \pom \to \R$ the solution $u = u_f$ to the Dirichlet problem with datum $f$ converges non-tangentially to $f$ $\sigma$-a.e. and
  \begin{align}
    \label{estimate:regularity} \|N_* (\nabla u)\|_{L^p(\partial \Omega)} \le C \|f\|_{\dot{W}^{1,p}(\pom)},
  \end{align}
  where $\|\cdot\|_{\dot{W}^{1,p}(\pom)}$ is the Haj{\l}asz seminorm (see Definition \ref{defin:hajlasz}). For $1 < p < \infty$, we say that the \emph{tangential regularity problem (for the Laplacian) is solvable in $L^p$} if the previous holds after we replace \eqref{estimate:regularity} with the estimate
  \begin{align*}
    \|N_* (\nabla u)\|_{L^p(\partial \Omega)} \le C \|\nablat f\|_{L^p(\pom)}.
  \end{align*}
\end{defin}

\section{John, local John and weak local John conditions}
\label{section:john_conditions}

In this section, we define different John type conditions, compare them by considering some examples and make some remarks related to literature. We assume that $\Omega \subset \R^{n+1}$ is an open set, with $n$-dimensional Ahlfors--David regular boundary. We do not assume that the corkscrew conditions hold in general but we discuss their role below.

\begin{defin}[Non-tangential paths]
  \label{defin:nt-paths}
  Let $\gamma \colon [0,1] \to \overline{\Omega}$ be a path in $\Omega$ from $X$ to $Y$. For $D \ge 1$, we say that $\gamma$ is a \emph{$D$-non-tangential path} if we have
  \begin{align*}
    \ell(\gamma(X,Z)) \le D \,\delta(Z)
  \end{align*}
  for every $Z \in \gamma$.
\end{defin}

Notice that Definition \ref{defin:nt-paths} is not symmetric with respect to $X$ and $Y$. The general idea is that we use non-tangential paths to measure how well we can connect boundary points to certain points inside the space. We use the name non-tangential path to emphasize the connection between these paths and non-tangential convergence we discussed in Section \ref{section:notation}. Indeed, if there exists a $D$-non-tangential path $\gamma$ from $x \in \pom$ to $Y \in \Omega$, then, by definition, we have
\begin{align*}
  |x-Z| \le \ell(\gamma(x,Z)) \le D \, \delta(Z)
\end{align*}
for every $Z \in \gamma$ and therefore we have $Z \in \Gamma_D(x)$ for every $Z \in \gamma$. Here $\Gamma_D(x)$ is the cone at $x$ of aperture $D$ (see \eqref{defin:cone}).

\begin{defin}[John, local John and Weak local John conditions]
  \label{defin:john_conditions}
  Let $D_0 \ge 1$. We say that $\Omega$ satisfies
  \begin{enumerate}
    \item[i)] the \emph{$D_0$-John condition} if there exists a point $X_0 \in \Omega$ such that for every $Y \in \Omega$ there exists a $D_0$-non-tangential path in $\Omega$ from $Y$ to $X_0$,

    \item[ii)] the \emph{local $D_0$-John condition} if for every $x \in \pom$ and every $r \in (0,\diam(\pom))$ there exists a point $Y_x \in B(x,r) \cap \Omega$ (that we call a \emph{local John point}) such that $B(Y_x, r/D_0) \subset \Omega$ and for every $z \in \Delta(x,r)$ there exists a $D_0$-non-tangential path $\gamma_z$ in $\Omega$ from $z$ to $Y_x$ such that $\ell(\gamma_z) \le D_0 r$,
    
    \item[iii)] the \emph{weak local $D_0$-John condition} if there exist constants $\theta \in (0,1]$ and $R \ge 2$ such that for every $X \in \Omega$ there exists a Borel set $F \subset \Delta_X \coloneqq B(X,R \, \delta(X)) \cap \pom$ such that $\sigma(F) \ge \theta \sigma(\Delta_X)$ and for every $z \in F$ there exists a $D_0$-non-tangential path $\gamma_z$ in $\Omega$ from $z$ to $X$ such that $\ell(\gamma_z) \le D_0 R \, \delta(X)$.
  \end{enumerate}
\end{defin}

The John condition was first used in the work of John \cite{john} but the terminology originates from the work of Martio and Sarvas \cite{martiosarvas}. The local John condition was first used in the work of Hofmann, Mitrea and Taylor \cite[Definition 3.12]{hofmannmitreataylor} and weak local John condition originates from the recent work of Azzam, Hofmann, Martell, Mourgoglou and the second named author \cite[Definition 2.11]{azzametal1}.

Generally, the John condition does not imply the local John condition, the local John condition does not imply the John condition, and there are domains that satisfy the weak local John condition but not the local John condition. We can see this by considering some straightforward examples:

\begin{example}
  \label{example:john_conditions}
   In $\R^2$ we have the following.
  \begin{enumerate}
    \item[(1)] $\Omega_1 \coloneqq B(0,1) \setminus \{(x,0) \colon x \in [0,1]\}$ satisfies the John condition (we can choose e.g. $X_0 = (0,-\tfrac{1}{2})$ as the ``John point'') but not the local John condition. We can see this by noticing that any ball $B((1,0),r) \cap \Omega_1$ contains points $(y,t) \in \Omega_1$ and $(z,t) \in \pom_1$ for both $t < 0$ and $t > 0$ with arbitrarily small $|t|$. Hence, no matter how we choose the point $\widetilde{X}_0 \in B((1,0),r) \cap \Omega_1$, there are points in $B((1,0),r) \cap \pom_1$ that can be connected to $\widetilde{X}_0$ inside $\Omega_1$ only with paths $\gamma$ satisfying $\ell(\gamma) \ge 1$.
    
    \item[(2)] $\Omega_2 \coloneqq \R^2 \setminus \partial B(0,1)$ satisfies the local John condition but not the John condition because we cannot connect a point in $B(0,1)$ to a point in $\R^2 \setminus B(0,1)$ with a path in $\Omega_2$.
    
    \item[(3)] $\Omega_3 \coloneqq B(0,1) \setminus \{(x,0) \colon x \in (-1,1)\}$ satisfies the weak local John condition but not the John condition or the local John condition. This is because $\Omega_3$ is not connected (and hence cannot satisfy the John condition) and it cannot satisfy the local John condition for the same reason why $\Omega_1$ in the part (1) of this example does not satisfy it.
  \end{enumerate}
\end{example}

To the best knowledge of the authors, it is not known if the local John condition alone implies the weak local John condition. We note that if there exists an open set $\Omega \subset \R^{n+1}$ with $n$-ADR boundary such that it satisfies the weak local John condition but not the local John condition, then $\Omega$ cannot satisfy the exterior corkscrew condition (see Lemma \ref{lemma:corkscrews_harmonic_measure}).

The John condition can be seen as a stronger form of connectivity of the space, but it does not imply connectivity for the boundary, not even if the exterior corkscrew condition holds. We can see this by considering the annulus $A \coloneqq B(0,2) \setminus B(0,1)$. By the same example and the set $\Omega_2$ in Example \ref{example:john_conditions} (2), we see that the local John condition does not generally imply connectivity for the space nor the boundary. However, by Theorem \ref{theorem:local_john_harnack}, we know that if the exterior corkscrew condition holds and the boundary of the space is Ahlfors--David regular, the local John condition implies connectivity for the space (but not the boundary, as we saw from the annulus $A$). By Corollary \ref{corollary:local_john_cad} and Corollary \ref{corollary:annular_quasiconvexity}, we know that the $2$-sided local John condition combined with Ahlfors--David regularity of the boundary implies annular quasiconvexity (and therefore connectivity) for the boundary.

We note that the local John condition in Definition \ref{defin:john_conditions} is the ``extreme'' case in the definition given in \cite[Definition 3.12]{hofmannmitreataylor}. To be more precise, let us consider the following weaker version of the local John condition:

\begin{defin}
  \label{defin:local_local_john}
  Let $\Omega \subset \R^{n+1}$ be an open set, $D_0 \ge 1$ and $0 < R_0 \le \diam(\pom)$. We say that $\Omega$ satisfies the \emph{local $(D_0,R_0)$-John condition} if for every $x \in \pom$ and every $r \in (0,R_0)$ there exists a point $Y_x \in B(x,r) \cap \Omega$ (that we call a \emph{local John point}) such that $B(Y_x, r/D_0) \subset \Omega$ and for every $z \in \Delta(x,r)$ there exists a $D_0$-non-tangential path $\gamma_z$ in $\Omega$ from $z$ to $Y_x$ such that $\ell(\gamma_z) \le D_0 r$.
\end{defin}
Thus, the local $D_0$-John condition and the local $(D_0,\diam(\pom))$-condition are the same thing. The reason why we consider only the case $(D_0,\diam(\pom))$ in Theorem \ref{theorem:local_john_harnack} is simply because the result may fail if the local John condition is not good enough. We see this by the following simple example (see also Figure \ref{figure:annulus}):

\begin{example}
  \label{example:local_local_john}
  Let $\Omega \coloneqq B(0,1) \cup \{X \in \R^2 \colon |X| > 2\} \subset \R^2$ be the interior of the complement of the annulus $B(0,2) \setminus B(0,1)$. Now $\pom$ is $1$-ADR and $\Omega$ satisfies the $2$-sided local $(D_0,\tfrac{1}{2})$-John condition for suitable $D_0 > 1$, but $\Omega$ is not a chord-arc domain. In addition, the boundary $\pom$ is not connected.
\end{example}

\begin{figure}[ht]
  \includegraphics[scale=0.46]{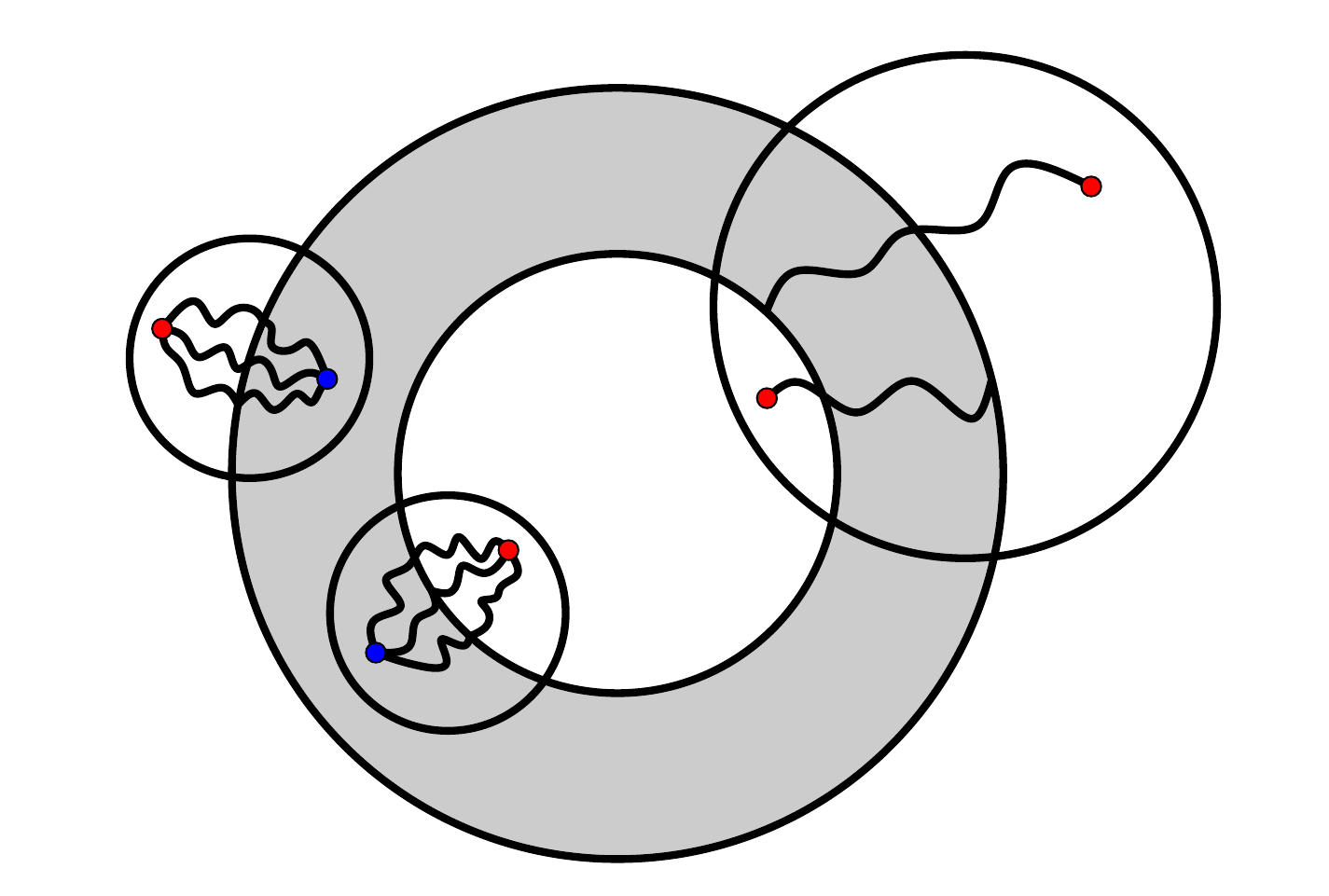}
  \caption{The $2$-sided local $(D_0,R_0)$-John condition does not imply Harnack chain condition or any kind of connectivity if $R_0$ is not large enough. In Example \ref{example:local_local_john}, for the complement of a closed annulus, we can find local John points for small balls centered at either the inner or the outer boundary. However, if the ball is large enough, then it contains both inner and outer boundary points, and we cannot find a point in $\Omega$ that connects to all the boundary points inside the ball without passing through the annulus itself.}
  \label{figure:annulus}
\end{figure}

In the proof of Theorem \ref{theorem:local_john_harnack}, we use harmonic measure theory and techniques that require that harmonic measure belongs to the class $\text{weak-}A_\infty(\sigma)$. By the main result of \cite{azzametal1}, we know that, in our context, the $\text{weak-}A_\infty$ property is equivalent with uniform rectifiability for $\pom$ and the weak local John condition for $\Omega$. We note that although we do not assume the weak local John condition, it follows from the assumptions of Theorem \ref{theorem:local_john_harnack}. Indeed, by \cite[p.~842]{davidjerison} (see also \cite{semmes_analysis}), we know that the $2$-sided corkscrew condition implies the \emph{interior big pieces of Lipschitz graphs condition} (see \cite[p.~572]{bennewitzlewis} for definition). This condition then implies that harmonic measure is in $\text{weak-}A_\infty$ by \cite[Theorem~1]{bennewitzlewis}. Combining this with \cite[Theorem 1.1]{azzametal1} gives us the weak local John condition:

\begin{lemma}
  \label{lemma:corkscrews_harmonic_measure}
  Let $\Omega \subset \R^{n+1}$ be an open set satisfying the $2$-sided corkscrew condition, with $n$-dimensional Ahlfors--David regular boundary. Then $\omega \in \text{weak-}A_\infty(\sigma)$ and $\Omega$ satisfies the weak local John condition.
\end{lemma}

\section{The Bilateral Weak Geometric Lemma and non-tangential approach}
\label{section:bwgl}

In this section, we consider some tools that we use in the proof of Theorem \ref{theorem:local_john_harnack} to connect two pieces of different paths to each other without going too close to the boundary. Our approach is based on the use of $\beta$-numbers of Jones \cite{jones} and the Bilateral Weak Geometric Lemma of David and Semmes \cite{davidsemmes_analysis}. Throughout the section, we assume that $\Omega \subset \R^{n+1}$ is an open set satisfying the $2$-sided corkscrew condition, with uniformly rectifiable boundary $\pom$, and $\D$ is a dyadic system on $\pom$. Recall that $\sigma \coloneqq \Hc^n|_{\pom}$ is the surface measure and $\delta(X) \coloneqq \dist(X,\pom)$ for $X \in \R^{n+1}$.

Given a ball $B = B(x,r) \subset \R^{n+1}$, a hyperplane $L \subset \R^{n+1}$ and a constant $c > 0$, we denote
\begin{align*}
  b\beta_{\pom}(B,L) \coloneqq \sup_{y \in B \cap \pom} \frac{\dist(y,L)}{r} + 
  \sup_{Y \in L \cap B} \frac{\delta(Y)}{r}
\end{align*}
and
\begin{align*}
  \Uc_{c}(L) \coloneqq \{Y \in \R^{n+1} \colon \dist(Y,L) < c\}.
\end{align*}
For any ball $B \subset \R^{n+1}$, we define the \emph{bilateral $\beta$-number} as
\begin{align*}
  b\beta_{\pom}(B) \coloneqq \inf_L b\beta_{\pom}(B,L),
\end{align*}
where the infimum is taken over all hyperplanes $L \subset \R^{n+1}$.

Recall from Notation \ref{notation:dyadic_cubes} that for a cube $Q \in \D$, we denote $B_Q = B(x_Q,\ell(Q))$, where $x_Q$ is the center of $Q$ and $\ell(Q)$ its side length. By a straightforward reformulation of \cite[Chapter I.2, Theorem 2.4]{davidsemmes_analysis}, the following version of the \emph{Bilateral Weak Geometric Lemma (BWGL)} holds:

\begin{lemma}
  \label{lemma:bwgl}
  For every $\eps > 0$, there exists a constant $C_\eps \ge 1$ such that
  \begin{align}
    \label{ineq:bwgl} \sum_{\substack{Q \in \D, Q \subset R,\\ b\beta_{\pom}(2B_Q)>\eps}} \sigma(Q) \le C_\eps \sigma(R)
  \end{align}
  for any $R \in \D$, i.e., for any $\eps > 0$ the collection $\{Q \in \D \colon b\beta_{\pom}(2B_Q)>\eps\}$ satisfies a Carleson packing condition with Carleson packing norm depending only on $\eps$, $n$ and uniform rectifiability constants.
\end{lemma}

The BWGL actually characterizes the uniform rectifiability property but we only need the part written in Lemma \ref{lemma:bwgl}.

In the proof of Theorem \ref{theorem:local_john_harnack}, we use the BWGL combined with the following lemmas that help us with technicalities related to constructing a path from a point to a nearby local John point. Alternatively, we could use the Whitney region constructions of Hofmann, Martell and Mayboroda \cite[Section 3]{hofmannmartellmayboroda} (and their straightforward geometric applications in \cite{hofmanntapiola_varopoulos} and \cite{hofmanntapiola_approx}) for the same purpose, but this alternative approach is slightly less elementary than the one we present in this paper.

\begin{lemma}
  \label{lemma:components}
  There exists $\eps_0 > 0$, depending only on the $2$-sided corkscrew condition, such that the following holds: if $B = B(x,r)$ is a ball with $x \in \pom$ and $L_B$ is a hyperplane such that $b\beta_{\pom}(B, L_B) < \eps \le \eps_0$, then $B \setminus \Uc_{\eps r}(L_B)$ consists of two convex components, $B^+$ and $B^-$, such that
  \begin{align*}
    B^+\subset \Omega \quad \text{ and }\quad B^-\subset \R^{n+1}\setminus \overline\Omega.
  \end{align*}
\end{lemma}

\begin{proof}
  Let $B = B(x,r)$ be a ball with $x \in \pom$, $\eps > 0$ and $L_B$ be a hyperplane such that $b\beta_{\pom}(B, L_B) < \eps$. By the definitions of $b\beta_{\pom}(B, L_B)$ and $\Uc_{\eps r}(L_B)$, we know that
  \begin{align*}
    \pom \cap B \subset \Uc_{\eps r}(L_B),
  \end{align*}
  and $B \setminus \Uc_{\eps r}(L_B)$ has exactly two components if $\eps$ is small enough, say $\eps < 1/10$. Thus, the two connected components $V_1$ and $V_2$, of $B \setminus \Uc_{\eps r}(L_B)$ are contained in $\R^{n+1} \setminus \pom =\Omega \cup (\R^{n+1}\setminus\overline\Omega)$. Furthermore, the components $V_1$ and $V_2$ are convex by the definition of $\Uc_{\eps r}(L_B)$, and each $V_i$ is either fully contained in $\Omega$ or fully contained in $\R^{n+1} \setminus \overline\Omega$. Indeed, if $V_i$ intersects both $\Omega$ and $\R^{n+1} \setminus \overline{\Omega}$, then the line segment from any point $Z_1 \in V_i \cap \Omega$ to any point $Z_2 \in V_1 \cap \R^{n+1} \setminus \overline{\Omega}$ has to contain a point $z_0 \in \pom$ which then has to belong to $V_i$ by convexity. This is impossible because $V_i \subset B \setminus \Uc_{\eps r}(L_B)$ and therefore $V_i \cap \pom = \emptyset$.

  Let us then show that if $\eps$ is small enough, one of the components $V_i$ lies in $\Omega$ and the other lies in $\R^{n+1} \setminus \overline{\Omega}$. By the $2$-sided corkscrew condition (applied for the surface ball $\Delta(x,r) = B(x,r) \cap \pom$), there exist balls
  \begin{align*}
    B_1 \coloneqq B(Z^+, cr)\subset \Omega \cap B,\qquad B_2 \coloneqq B(Z^-, cr) \subset (\R^{n+1} \setminus \overline{\Omega}) \cap B,
  \end{align*}
  where $c \in (0,1)$ is independent of $\Delta$. Let us assume that $\eps \le c/5$. Now neither $B_1$ nor $B_2$ can be contained in $\Uc_{\eps r}(L_B)$ and therefore both the balls intersect $V_1 \cup V_2$. Since each $V_i$ is either fully contained in $\Omega$ or fully contained in $\R^{n+1} \setminus \overline{\Omega}$, we know that $V_i \cap B_1 \neq \emptyset$ implies $V_i \subset \Omega$ and $V_i \cap B_2 \neq \emptyset$ implies $V_i \subset \R^{n+1} \setminus \overline{\Omega}$. We also notice that if $V_1 \cap B_1 \neq \emptyset$, then $V_1 \cap B_2 = \emptyset$ since otherwise $V_1$ intersects both $\Omega$ and $\R^{n+1} \setminus \overline{\Omega}$. Thus, since both $B_1$ and $B_2$ intersect $V_1 \cup V_2$, we know that $V_1 \cap B_1 \neq \emptyset$ implies $V_2 \cap B_2 \neq \emptyset$, and similarly $V_1 \cap B_2 \neq \emptyset$ implies $V_2 \cap B_1 \neq \emptyset$. In particular, $B_1$ intersects exactly one of the components $V_i$, say $V_1$, which is then contained in $\Omega$, and $B_2$ then intersects $V_2$, which is contained in $\R^{n+1} \setminus \overline{\Omega}$. Thus, we may set $B^+ = V_1$ and $B^- = V_2$ and choose $\eps_0 = \min\{1/10, c/5\}$. This completes the proof. 
\end{proof}

\begin{lemma}
  \label{lemma:path}
  Let $B = B(x_0,r)$ be a ball with $x_0 \in \pom$, $X_1 \in \Omega$ a point such that $|X_1 - x_0| \approx \delta(X_1) \ge r/2$ and $\gamma$ a $D$-non-tangential path from $x_0$ to $X_1$, where $D \ge 1$. Let $\eps_0 > 0$ be as in Lemma \ref{lemma:components} and suppose that $0 < \eps < \min\{\eps_0,\tfrac{1}{12D}\}$. Let $L_B$ be a hyperplane such that $b\beta_{\pom}(B,L_B) < \eps$, and let $B^+$ and $B^-$ be the components of $B$ as in Lemma \ref{lemma:components}. Now $\gamma$ intersects $B^+ \setminus \Uc_{2\eps r}(L_B)$.
\end{lemma}

\begin{proof}
  Since $|X_1 - x_0| \ge r/2$, we know that $X_0 \notin B(x_0,r/4)$. Thus, there exists a point $Y_0 \in \gamma \cap \partial B(x_0,r/4)$. We claim that $Y_0 \in B^+ \setminus \Uc_{2 \eps r}(L_B)$.
  
  We notice that, by the definition of $D$-non-tangential paths, we have
  \begin{align}
    \label{estimate:y_0_point}
    \delta(Y_0) \ge \frac{1}{D} \ell(\gamma(x_0,Y_0)) \ge \frac{1}{D} |x_0 - Y_0| = \frac{1}{4D}r.
  \end{align}
  For any point $X \in \R^{n+1}$, let $\pr_X$ be its orthogonal projection onto $L_B$. Let $Z \in \tfrac{1}{2}B \cap \Uc_{2\eps r}(L_B)$. Now it holds that
  \begin{align*}
    |\pr_Z - x_0| \le |\pr_Z - Z| + |Z - x_0| < 2\eps r + \frac{1}{2}r < r,
  \end{align*}
  and thus, $\pr_Z \in B$. In particular, since $\pr_Z \in L_B$ and $b\beta_{\pom}(B,L_B) < \eps$, we have
  \begin{align}
    \label{estimate:z_point}
    \delta(Z) \le |Z - \pr_Z| + \delta(\pr_Z) \le 2\eps r + \eps r = 3\eps r < \frac{1}{4D} r
  \end{align}
  since $\eps < \tfrac{1}{12D}$. In particular, by \eqref{estimate:y_0_point} and \eqref{estimate:z_point}, we know that $Y_0 \notin \tfrac{1}{2} B \cap \Uc_{2\eps r}(L_B)$. On the other hand, since $Y_0 \in \gamma \cap \partial B(x_0,r/4)$, we know that $Y_0 \in B \cap \Omega$ and hence $Y_0 \in B^+$ by Lemma \ref{lemma:components}. In particular, $Y_0 \in \gamma \cap B^+ \setminus \Uc_{2\eps r}(L_B)$, which proves the claim.
\end{proof}

\begin{lemma}
  \label{lemma:subball}
  Let $\eps_0 > 0$ be as in Lemma \ref{lemma:components} and suppose that $0 < \eps < \tfrac{1}{8}\eps_0$. Let $B = B(x_0,r)$ be a ball with $x_0 \in \pom$, $L_B$ be a hyperplane such that $b\beta_{\pom}(B,L_B) < \eps$ and $y_0 \in \tfrac{1}{2} B\cap\pom$. Now we have
  \begin{enumerate}
    \item[$\bullet$] $b\beta_{\pom}(B(y_0,r/4),L_B) < 4 \eps < \tfrac{1}{2} \eps_0$,
    \item[$\bullet$] the set $\Omega \cap B(y_0,r/4) \setminus \Uc_{\eps r}(L_B)$ is convex,
    \item[$\bullet$] for any point $Y \in \Omega \cap B(y_0,r/4) \setminus \Uc_{2\eps r}(L_B)$ we have $\delta(Y) \ge \eps r$.
  \end{enumerate}
\end{lemma}

\begin{proof}
  The second claim follows from the first claim combined with Lemma \ref{lemma:components}, and the first claim follows from the definition of $\beta$-numbers and the facts $B(y_0,r/4) \subset B$ and $b\beta_{\pom}(B,L_B) < \eps < \tfrac{1}{8}\eps_0$ in a straightforward way:
  \begin{align*}
    b\beta_{\pom}(B,L_B)
    &= \sup_{y \in B \cap \pom} \frac{\dist(y,L_B)}{r} + \sup_{Y \in L_B \cap B} \frac{\delta(Y)}{r} \\
    &\ge  \frac{1}{4} \sup_{y \in B(y_0,r/4) \cap \pom} \frac{\dist(y,L_B)}{\frac{r}{4}} + \sup_{Y \in L_B \cap B(y_0,r/4)} \frac{\delta(Y)}{\frac{r}{4}}
    = \frac{1}{4} b\beta_{\pom}(B(y_0,r/4),L).
  \end{align*}
  For the third claim, let $Y \in \Omega \cap B(y_0,r/4) \setminus \Uc_{2\eps r}(L_B)$. Since $y_0 \in \pom$, we know that $\delta(Y) < \tfrac{1}{4}r$. On the other hand, for any point $Z \in \R^{n+1} \setminus B$ we have
  \begin{align*}
    r \le |x_0 - Z| \le |x_0 - y_0| + |y_0 - Y| + |Y - Z| < \frac{1}{2}r + \frac{1}{4}r + |Y - Z|,
  \end{align*}
  and thus, $|Y - Z| > \tfrac{1}{4}r$. In particular, we have $\delta(Y) = \inf_{z \in \pom \cap B} |Y - z|$. Let $z_Y \in \pom \cap B$ be a point such that $\delta(Y) = |Y - z_Y|$. Since $Y \notin \Uc_{2\eps r}(L_B)$, we know that $\dist(Y,L_B) \ge 2\eps r$, and since $b\beta_{\pom}(B,L_B) < \eps$, we know that $\dist(z_Y,L_B) < \eps r$. In particular,
  \begin{align*}
    2\eps r \le \dist(Y,L_B) \le |Y - z_Y| + \dist(z_Y,L_B) < |Y - z_Y| + \eps r = \delta(Y) + \eps r,
  \end{align*}
  and therefore $\delta(Y) \ge \eps r$, as claimed.
\end{proof}

\section{Local John and exterior corkscrews imply Harnack chains}
\label{section:local_john_harnack}

In this section, we prove Theorem \ref{theorem:local_john_harnack}, i.e., that local John condition together with exterior corkscrew condition implies the existence of Harnack chains. Let $\Omega \subset \R^{n+1}$ be an open set with $n$-ADR boundary $\pom$, and suppose that $\Omega$ satisfies the local $D_1$-John condition and exterior corkscrew condition. Througout the section, $\D$ is a dyadic system on $\pom$.

\begin{proof}[Proof of Theorem \ref{theorem:local_john_harnack}]
  Let $X,Y \in \Omega$ with $\delta(X), \delta(Y) \ge \rho > 0$ and $|X-Y| < \Lambda \rho$ for $\Lambda \ge 1$. We start by noticing that by Theorem \ref{theorem:corkscrew_ur} we know that $\pom$ is UR, and by Lemma \ref{lemma:corkscrews_harmonic_measure} we know that harmonic measure belongs to the class $\text{weak-}A_\infty(\sigma)$.
  
  We will construct a path between $X$ and $Y$ that stays far away from the boundary and use this path to construct the Harnack chain between $X$ and $Y$. To ensure that we stay away from the boundary in a quantitative way, we have to be careful with the construction. This will make the construction quite technical and therefore we divide the proof into a few different parts.
  
  The argument we present below works as it is in the case $\diam(\pom) = \infty$. We discuss the other cases in the end of the proof.
  
  \textbf{Part 1:} \emph{choosing suitable cubes for $X$ and $Y$ for Lemma \ref{lemma:connective_set}.} Take a point $z_X \in \pom$ such that $\delta(X) = |X - z_X|$. Let $c_0$ and $C_0$ be the constants from Lemma \ref{lemma:bourgain}. Now, by Lemma \ref{lemma:bourgain}, since $X \in B(z_X, c_0 \cdot 2\delta(X)/c_0)$, we know that $\omega^X(\Delta(z_X, 2\delta(X)/c_0)) \ge 1/C_0$. Let us cover $\Delta(z_X, 2\delta(X)/c_0)$ by dyadic cubes $Q_i$ of the same side lengths such that $\ell(Q_i) \approx 2\delta(X)/c_0$ and $X \notin 4B_{Q_i}$ for any $i$. There are at most $c_n$ of these types of cubes $Q_i$. Since $\omega^X(\Delta(z_X, 2\delta(X)/c_0)) \ge 1/C_0$ and the cubes $Q_i$ cover $\Delta(z_X, 2\delta(X)/c_0)$, we know that there exists a cube $Q_X \in \{Q_i\}_i$ such that $\omega^X(Q_X) \ge (c_nC_0)^{-1}$. In addition, the cube $Q_X$ satisfies $\ell(Q_X) \approx \delta(X) \approx \dist(X,Q_X)$ with uniformly bounded implicit constants since
  \begin{enumerate}
    \item[$\bullet$] one of the cubes $Q_i$ contains the point $z_X$ and $\delta(X) = |X - z_X|$,
    \item[$\bullet$] there is only a uniformly bounded number of the cubes $Q_i$, and
    \item[$\bullet$] $\ell(Q_i) \approx 2\delta(X)/c_0$ for every $i$.
  \end{enumerate}
  Similarly, we can choose a cube $Q_Y$ that has the same properties but with respect to $Y$ instead of $X$.
  
  \textbf{Part 2:} \emph{choosing a local John point.} Let us consider the ball $B(z_X,D_1 r_0)$ where
  \begin{align}
    \label{defin:radius_r0} r_0 \coloneqq C_n \cdot \max\{\delta(X),\delta(Y),|X-Y|\}
  \end{align}
  for a large enough dimensional constant $C_n$ such that $Q_X, Q_Y \subset B(z_X, D_1 r_0)$. Let $Z_0 \in \Omega$ be a local John point for the ball $B(z_X, D_1 r_0)$. By the local John condition, we know that $B(Z_0, D_1 r_0/D_1) = B(Z_0,r_0) \subset \Omega$ and there exists $D_1$-non-tangential paths from the points of $\Delta(z_X, D_1 r_0)$ to $Z_0$. In particular, there exist $D_1$-non-tangential paths from the points on $Q_X$ and $Q_Y$ to $Z_0$.
  
  \textbf{Part 3:} \emph{choosing starting points for paths.} Let us consider the cube $Q_X$. By the choice of $Q_X$ and Lemma \ref{lemma:connective_set}, we know that there exist constants $\alpha \in (0,1]$ and $D_2 \ge 1$ (independent of $X$ and $Q_X$) and a subset $A_X \subset Q_X$ such that
  \begin{enumerate}
    \item[$\bullet$] $\sigma(A_X) \ge \alpha \, \sigma(Q_X)$, and
    \item[$\bullet$] there exist $D_2$-non-tangential paths from the points on $A_X$ to $X$.
  \end{enumerate}
  Let $\eps_0 > 0$ be as in Lemma \ref{lemma:components} and set
  \begin{align}
    \label{defin:choice_of_eps} \eps \coloneqq \frac{1}{8}\min\left\{\frac{1}{10}\eps_0, \frac{1}{12D_1}, \frac{1}{12D_2}\right\}.
  \end{align}
  By the Bilateral Weak Geometric Lemma (Lemma \ref{lemma:bwgl}) and Lemma \ref{lemma:subset_carleson}, we know that there exists a cube $R_X$ such that $b\beta_{\pom}(2B_{R_X}) < \eps$, $\sigma(R_X \cap A_X) > 0$ and $\ell(R_X) \approx \ell(Q_X)$ where the implicit constant depends only on $\alpha$ and the Carleson packing norm of the collection of cubes $Q$ such that $b\beta_{\pom}(2B_Q) > \eps$. Let us choose any point $\widetilde{z}_X \in R_X \cap A_X$. Similarly, we can choose sets $A_Y$ and $R_Y$ and a point $\widetilde{z}_Y$ for the point $Y$.
  
  \textbf{Part 4:} \emph{constructing paths between the local John point and $X$ and $Y$.} Let us recap:
  \begin{enumerate}
    \item[$\bullet$] $R_X \subset \pom$ is a dyadic cube such that $\ell(R_X) \approx \delta(X)$,
    \item[$\bullet$] $\eps > 0$ is a number defined in \eqref{defin:choice_of_eps} and we have $b\beta_{\pom}(2B_{R_X}) < \eps$,
    \item[$\bullet$] $\widetilde{z}_X \in R_X$ is a point such that there exists a $D_2$-non-tangential path from $\widetilde{z}_X$ to $X$,
    \item[$\bullet$] $Z_0 \in \Omega$ is the local John point for the ball $B(z_X,D_1 r_0)$, where $|X - z_X| = \delta(X)$ and $r_0 > 0$ is the radius defined in \eqref{defin:radius_r0}, and $\widetilde{z}_X \in B(z_X,D_1 r_0)$.
  \end{enumerate}
  Let $\gamma_1$ be a $D_1$-non-tangential path from $\widetilde{z}_X$ to $Z_0$ and $\gamma_2$ be a $D_2$-non-tangential path from $\widetilde{z}_X$ to $X$. Let $L_X$ be a hyperplane such that $b\beta_{\pom}(2B_{R_X},L_X) < \eps$ and consider the ball $B(\widetilde{z}_X,\ell(R_X)/2)$. Since we know that $\widetilde{z}_X \in R_X \subset B_{R_X} \subset 2B_{R_X}$, Lemma \ref{lemma:subball} gives us
  \begin{align*}
    b\beta_{\pom}(B(\widetilde{z}_X,\ell(R_X)/2),L_X) < 4\eps < \frac{1}{2} \min\left\{\frac{1}{10}\eps_0, \frac{1}{12D_1}, \frac{1}{12D_2}\right\}.
  \end{align*}
  By Lemma \ref{lemma:path} (applied for $b\beta_{\pom}(B(\widetilde{z}_X,\ell(R_X)/2),L_X) < 4\eps$), we know that both $\gamma_1$ and $\gamma_2$ intersect
  \begin{align*}
    B(\widetilde{z}_X,\ell(R_X)/2) \setminus \Uc_{2 \cdot 4\eps \cdot \ell(R_X)/2}(L_X) =  B(\widetilde{z}_X,\ell(R_X)/2) \setminus \Uc_{2\eps \cdot 2\ell(R_X)}(L_X).
  \end{align*}
  Let $Z_1 \in \gamma_1$ and $X_1 \in \gamma_2$ be any points such that $Z_1, X_1 \in B(\widetilde{z}_X,\ell(R_X)/2) \setminus \Uc_{2\eps \cdot 2\ell(R_X)}(L_X)$, and let $\gamma_3$ be the line segment connecting $Z_1$ to $X_1$. By Lemma \ref{lemma:subball}, we know that $\gamma_3$ is fully contained in $B(\widetilde{z}_X,\ell(R_X)/2) \setminus \Uc_{2\eps \cdot 2\ell(R_X)}(L_X)$ and we have
  \begin{align}
    \label{estimate:on_gamma3} \delta(\widehat{X}) \ge \eps \cdot 2\ell(R_X) \gtrsim \delta(X)
  \end{align}
  for every $\widehat{X} \in \gamma_3$, where the implicit constant depends on $\eps_0$, $D_1$, $D_2$ and the structural constants appearing in the proof. Since $\gamma_3$ is line segment that is fully contained in $B(\widetilde{z}_X,\ell(R_X)/2) \setminus \Uc_{2\eps \cdot 2\ell(R_X)}(L_X)$, we know that
  \begin{align}
    \label{estimate:length_gamma3} \ell(\gamma_3) \le \ell(R_X) \lesssim \delta(X).
  \end{align}
  See Figure \ref{figure:path_construction} for an illustration of the situation.
  
  \begin{figure}[ht]
    \includegraphics[scale=0.75]{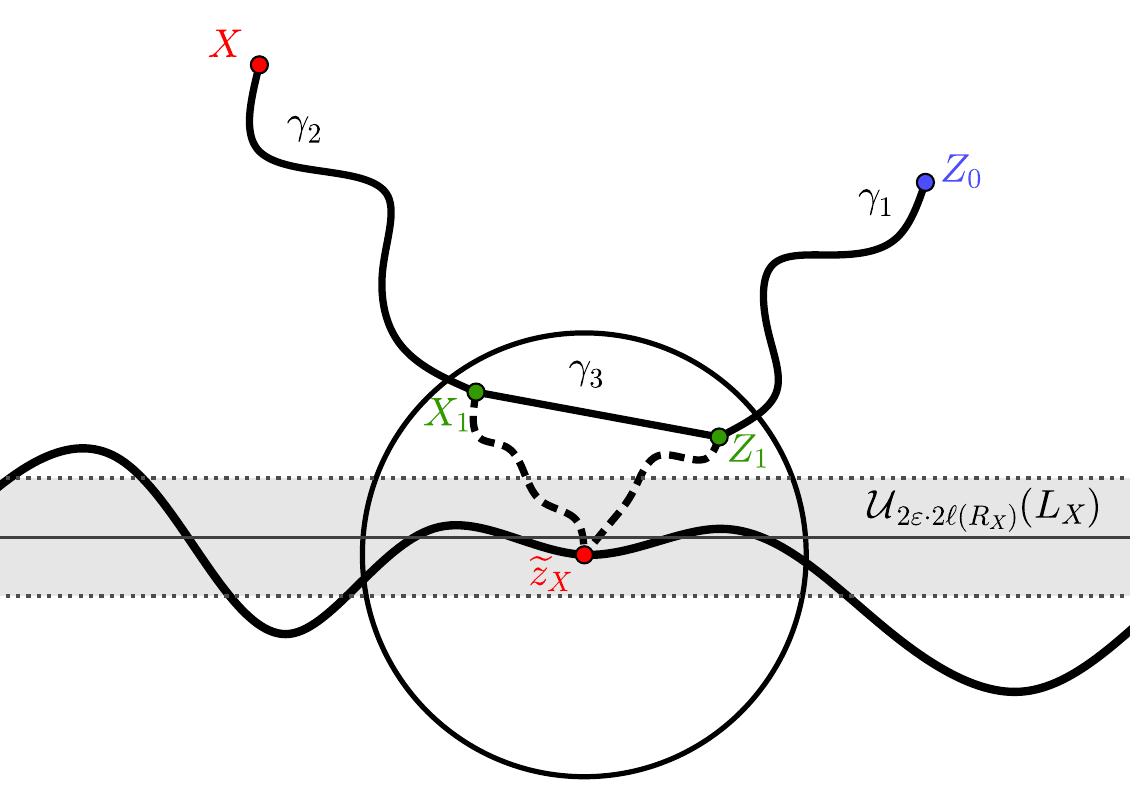}
    \caption{We construct a path between $X$ and the local John point $Z_0$ that stays far away from the boundary the following way. By a careful choice of a point $\widetilde{z}_X \in \pom$, we know that there exists a non-tangential path $\gamma_1$ from $\widetilde{z}_X$ to $Z_0$ (given by the local John condition) and a non-tangential path $\gamma_2$ from $\widetilde{z}_X$ to $X$ (given the weak-$A_\infty$ property of harmonic measure through Lemma \ref{lemma:connective_set}). By the BWGL, there exists a hyperplane $L_X$ such that $L_X$ approximates $B \cap \pom$ well for a suitable $\eps > 0$ and a ball $B = B(\widetilde{z}_X,\ell(R_X)/2)$, where $R_X$ is a dyadic cube containing $\widetilde{z}_X$ such that $\ell(R_X) \approx \delta(X)$. By applications of the BWGL, we know that there exist points $Z_1 \in \gamma_1 \cap B \cap \Omega$ and $X_1 \in \gamma_2 \cap B \cap \Omega$ that do not lie on the strip $\mathcal{U}_{2 \eps \cdot 2\ell(R_X)}(L_X)$, i.e., they lie reasonably far away from the boundary. Because $\Omega \cap B \setminus \Uc_{2\eps \cdot 2\ell(R_X)}(L_X)$ is convex, we can connect $Z_1$ and $X_1$ to each other with a line segment $\gamma_3$. We can now travel from $Z_0$ to $X$ by using pieces of the paths $\gamma_1$, $\gamma_3$ and $\gamma_2$.}
    \label{figure:path_construction}
  \end{figure}

  We now build a path $\gamma_X$ from $Z_0$ to $X$ by glueing together (after rescaling) the reversed part of $\gamma_1$ that travels from from $Z_1$ to $Z_0$, the whole $\gamma_3$ (from $Z_1$ to $X_1$) and the part of $\gamma_2$ that travels from $X_1$ to $X$. Similarly, we can choose a hyperplane $L_Y$ for $R_Y$ and points $Z_2, Y_1 \in B(\widetilde{z}_Y,\ell(R_Y)/2) \setminus \Uc_{2\eps \cdot 2\ell(R_Y)}(L_Y)$ and construct a path $\gamma_Y$ from $Z_0$ to $Y$ that passes through $Z_2$ and $Y_1$.
  
  \textbf{Part 5:} \emph{constructing the Harnack chains.} Let us consider the path $\gamma_X$. Since $\gamma_2$ is a non-tangential path (from $\widetilde{z}_X$ to $X$) and $\delta(X_1) \approx \delta(X)$, we know that $\ell(\gamma_2) \lesssim \delta(X)$ and $\delta(\widehat{X}) \approx \delta(X)$ for every $\widehat{X} \in \gamma_2(X_1,X)$. By \eqref{estimate:length_gamma3}, we know that $\ell(\gamma_3) \lesssim \delta(X)$, and by \eqref{estimate:on_gamma3}, $\delta(\widehat{X}) \approx \delta(X)$ for every $\widehat{X} \in \gamma_3$. Thus, for a suitable uniform implicit constant, we may cover $\gamma_X \setminus \gamma_1$ by a uniformly bounded number of balls $B_i$ with radii $r_i \approx \delta(X)$ satisfying $\dist(B_i,\pom) \approx \diam(B_i)$. As for $\gamma_1$, we notice that since $\gamma_1$ is a non-tangential path (from $\widetilde{z}_X$ to $Z_0$) given by the local John condition and $\delta(Z_1) \approx \delta(X)$, we have
  \begin{align*}
    \ell(\gamma_1(\widetilde{z}_X,Z_1)) \lesssim \delta(Z_1) \approx \delta(X) \ \ \text{ and } \ \ \ell(\gamma_1) \lesssim r_0 = C_n \cdot \max\{\delta(X),\delta(Y),|X-Y|\},
  \end{align*}
  and therefore $\delta(\widehat{X}) \gtrsim \delta(X)$ for all $\widehat{X} \in \gamma_1(Z_1,Z_0)$.
  In particular, we can cover $\gamma_X \setminus (\gamma_2 \cup \gamma_3)$ by $N_0 \lesssim r_0 / \delta(X)$ balls $B_k$ of radii $r_k \approx \delta(X)$ such that $\dist(B_k,\pom) \approx \diam(B_k)$ for each $k$. Recall that $\delta(X), \delta(Y) \ge \rho > 0$ and $|X-Y| < \Lambda \rho$ for $\Lambda \ge 1$. Let us consider different cases:
  \begin{enumerate}
    \item[$\bullet$] Suppose that $r_0 \lesssim |X-Y|$. Now we have
    \begin{align*}
      N_0 \lesssim \frac{|X-Y|}{\delta(X)} = \frac{\rho}{\delta(X)} \Lambda \le \Lambda.
    \end{align*}
    
    \item[$\bullet$] Suppose that $r_0 = C_n \delta(X)$. Then
    \begin{align*}
      N_0 \lesssim \frac{C_n \delta(X)}{\delta(X)} = C_n \le C_n \Lambda
    \end{align*}
    
    \item[$\bullet$] Suppose that $r_0 = C_n \delta(Y)$ and $\delta(Y) \gg |X-Y|$. Then, by the triangle inequality, $\delta(X) \approx \delta(Y)$ and
    \begin{align*}
      N_0 \lesssim \frac{C_n \delta(Y)}{\delta(X)} \approx \frac{C_n \delta(X)}{\delta(X)} = C_n \le C_n \Lambda.
    \end{align*}
  \end{enumerate}
  By almost identical arguments, we know that same estimates hold $\gamma_Y$. Thus, we can connect $X$ to $Y$ by taking the chain of balls $B_i$ that covers $\gamma_X$ and $\gamma_Y$. These balls satisfy $\dist(B_i,\pom) \approx \diam(B_i)$ for every $i$ with possibly different implicit constants. We may choose a constant $\widetilde{C} \ge 1$ such that
  \begin{align*}
    \frac{1}{\widetilde{C}} \diam(B_i) \le \dist(B_i,\pom) \le \widetilde{C} \diam(B_i)
  \end{align*}
  for every $i$ since we used only a finite number of different constants in the construction (six, to be more precise, since we built $\gamma_X$ and $\gamma_Y$ using three pieces for each path with uniform implicit constants for each piece). We needed a uniformly bounded number of balls to cover the two out of three pieces of $\gamma_X$ (same for $\gamma_Y$) and $N \lesssim \Lambda$ balls to cover the last piece of $\gamma_X$ (same for $\gamma_Y$). Thus, the number of balls $B_i$ is bounded by $\Cc \Lambda$, where $\Cc$ is a large enough constant depending only on $n$ and the ADR, UR, local John, weak-$A_\infty$ and corkscrew constants. Hence, $(B_i)_i$ is a Harnack chain between $X$ and $Y$. This completes the proof for the case $\diam(\pom) = \infty$.
  
  \
  
  Let us then consider the remaining two cases. Suppose that $\diam(\pom) < \infty$ and $\diam(\Omega) < \infty$. Now, if $\max\{\delta(X),\delta(Y),|X-Y|\} \ll \diam(\pom)$, things work just as earlier. Thus, we may assume that $\max\{\delta(X),\delta(Y),|X-Y|\} \approx \diam(\pom)$. Since $\diam(\pom) < \infty$, there exists a point $Z_0 \in \Omega$ such that for any $z \in \pom$ there exists a $D_1$-non-tangential path from $z$ to $Z_0$. Now the previous proof works when we simply choose this ''global`` local John point instead of the point we chose in Part 2.
  
  \
  
  Finally, suppose that $\diam(\pom) < \infty$ and $\diam(\Omega) = \infty$. Let us consider the following three cases:
  \begin{enumerate}
    \item[$\bullet$] If $\max\{\delta(X),\delta(Y),|X-Y|\} \ll \diam(\pom)$, we proceed as in the case ''$\diam(\pom) = \infty$``.
    
    \item[$\bullet$] If $\delta(X) \lesssim \diam(\pom)$, $\delta(Y) \lesssim \diam(\pom)$ and $|X-Y| \approx \diam(\pom)$, we proceed as in the case ''$\diam(\pom) < \infty$ and $\diam(\Omega) < \infty$``.
    
    \item[$\bullet$] If $\delta(X) \gg \diam(\pom)$ and $\delta(Y) \gg \diam(\pom)$, we can construct a Harnack chain from $X$ to $Y$ in the simple geometry of $\R^{n+1} \setminus B(Z,s)$ for $Z \in \Omega$ and $s \approx \diam(\pom)$.
  \end{enumerate}
   Thus, we may assume that $\delta(X) \lesssim \diam(\pom)$ and $\delta(Y) \gg \diam(\pom)$. The previous procedure does not work directly in this case because we cannot apply Lemma \ref{lemma:bourgain} and Lemma \ref{lemma:connective_set} for $\delta(Y) \gg \diam(\pom)$. Instead, we connect $Y$ to a point that is close enough to the boundary and then connect this point to $X$.
   
   Let $\widehat{y} \in \pom$ be a point such that $\delta(Y) = |Y - \widehat{y}|$. Let us consider the line segment $L$ with end points $\widehat{y}$ and $Y$. Since $L$ is a line segment and $\delta(Y) = |Y-\widehat{y}|$, we have $\delta(Z) = |Z - \widehat{y}|$ for any $Z \in L$. Let us take a point $\widehat{Y} \in L$ such that $\delta(\widehat{Y}) = \delta(X)$. Now we can use the earlier procedure to construct a Harnack chain $(B_i)_i$ from $X$ to $\widehat{Y}$, possibly using a ''global`` local John point as we did in the case ''$\diam(\pom) < \infty$ and $\diam(\Omega) < \infty$``. Since $\delta(X) = \delta(\widehat{Y}) \lesssim \diam(\pom)$ and $\delta(Y) \gg \diam(\pom)$, the length of the chain $(B_i)_i$ depends only on
   \begin{align*}
     \frac{|X - \widehat{Y}|}{\delta(X)} \lesssim \frac{|X - Y|}{\delta(X)} \le \frac{\rho}{\delta(X)} \Lambda \le \Lambda.
   \end{align*}
   We then continue this chain from $\widehat{Y}$ to $Y$ by covering the line segment from $\widehat{Y}$ to $Y$ using balls $\widehat{B}_k$ with radii $r_k \approx \delta(\widehat{Y}) = \delta(X)$ such that $\dist(\widehat{B}_k,\pom) \approx \diam(\widehat{B}_k)$. The number of balls $\widehat{B}_k$ that we need is approximately
  \begin{align*}
    \frac{\ell(L)}{\delta(X)}
    \approx \frac{\delta(Y)}{\delta(X)}
    \approx \frac{|X-Y|}{\delta(X)}
    \le \frac{\rho}{\delta(X)} \Lambda
    \le \Lambda,
  \end{align*}
  where we used that $\delta(X) \lesssim \diam(\pom)$ and $\delta(Y) \gg \diam(\pom)$ (and therefore we have $\delta(Y) \approx |X-Y|$). Hence, combining the chains $(B_i)_i$ and $(\widehat{B}_k)_k$ gives us a Harnack chain from $X$ to $Y$. This completes the proof of the last case.
\end{proof}

\section{Weak $(1,1)$ version of the Hofmann--Mitrea--Taylor Poincar\'e inequality, and quasiconvexity}
\label{section:hmt_and_quasiconvex}

Let $\Omega$ be a $2$-sided chord-arc domain. In this section, we consider some parts of the Hofmann--Mitrea--Taylor theory that we need for the proof of Theorem \ref{theorem:boundary_poincare}. Recall the definitions of the tangential gradient $\nablat f$, the Hofmann--Mitrea--Taylor Sobolev space $L^p_1(\pom)$ and the Hofmann--Mitrea--Taylor tangential derivatives $\partial_{\smallt,j,k} f$ and gradient $\nablahmt f$ in Section \ref{subsection:gradients}. It is straightforward to check that for a compactly supported Lipschitz function $f$ on $\pom$ we have $f \in L^p_1(\pom)$ for every $1 < p < \infty$ and hence, the Hofmann--Mitrea--Taylor gradient $\nabla_{\HMT} f$ exists.

\begin{remark}
  \label{remark:compact_support}
  In this section and in Section \ref{section:poincare}, we mostly consider compactly supported Lipschitz functions but this is enough for our purposes: by the results of Keith \cite{keith} (see Theorem \ref{theorem:keith} below) and Mourgoglou and the second named author \cite{mourgogloutolsa2} (see Lemma \ref{lemma:properties_of_hmt_objects} and Lemma \ref{lemma:norm_of_tangential_gradient} below), verifying Theorem \ref{theorem:boundary_poincare} for compactly supported Lipschitz functions implies Corollary \ref{corollary:poincare}, which then allows us to give up the assumption about compact support for Theorem \ref{theorem:boundary_poincare}. Assuming that our Lipschitz functions are compactly supported ensures that the Hofmann--Mitrea--Taylor gradients exist and we can use the machinery in \cite{hofmannmitreataylor} and \cite{mourgogloutolsa2} without additional considerations.
\end{remark}

Let us start by recalling two lemmas from \cite{mourgogloutolsa2}. First, when considering compactly supported Lipschitz functions, the norm of the Hofmann--Mitrea--Taylor gradient agrees with the norm of the tangential gradient almost everywhere:

\begin{lemma}[{\cite[Lemma 6.4]{mourgogloutolsa2}}]
  \label{lemma:properties_of_hmt_objects}
  Let $f$ be a compactly supported Lipschitz function on $\pom$. Then
  \begin{align}
    \label{identity:gradients} |\nablat f| = |\nablahmt f|
  \end{align}
  $\sigma$-almost everywhere.
\end{lemma}
In fact, \cite[Lemma 6.4]{mourgogloutolsa2} shows us that $\nablat f = -\nablahmt f$ almost everywhere, but we only need the comparability of the norms. We note that \cite[Lemma 6.4]{mourgogloutolsa2} is formulated for bounded domains but a routine inspection of the proof shows us that it holds for compactly supported Lipschitz functions also in unbounded domains.

Furthermore, for Lipschitz functions, the norm of the tangential gradient (and hence the norm of the Hofmann--Mitrea--Taylor gradient) agrees almost everywhere with the ''local Lipschitz constant`` function:

\begin{lemma}[{\cite[Lemma 2.2]{mourgogloutolsa2}}]
  \label{lemma:norm_of_tangential_gradient}
  Let $f$ be a Lipschitz function on $\pom$. Then
  \begin{align}
    \label{bound:norm_of_tangential_gradient}
    |\nablat f(x)| = \limsup_{\pom \ni y \to x} \frac{|f(x) - f(y)|}{|x-y|} \approx \limsup_{r \to 0} \fint_{\Delta(x,r)} \frac{|f(x) - f(y)|}{|x-y|} \, d\sigma(y)
  \end{align}
  for $\sigma$-a.e. $x \in \pom$.
\end{lemma}

As a part of their extensive work, Hofmann, Mitrea and Taylor proved the following weak $(p,p)$-Poincar\'e inequality with a tail:

\begin{theorem}[{\cite[Proposition 4.13]{hofmannmitreataylor}}]
  \label{theorem:hmt_poincare}
  Let $1 < p < \infty$. There exists a constant $C = C(\Omega,p)$ such that for every $f \in L^p_1(\pom)$, $x \in \pom$ and $r > 0$ we have
  \begin{align}
    \nonumber \left( \fint_{\Delta} |f - \langle f \rangle_{\Delta}|^p \, d\sigma \right)^{1/p} 
    & \le Cr \left( \fint_{5\Delta} |\nablahmt f|^p \, d\sigma \right)^{1/p} \\
    \label{ineq:hmt_poincare} &\quad + Cr \sum_{j=2}^{\infty} \frac{2^{-j}}{\sigma(2^j \Delta)} \int_{2^j \Delta \setminus 2^{j-1} \Delta} |\nablahmt f| \, d\sigma,
  \end{align}
  where $\Delta \coloneqq \Delta(x,r)$.
\end{theorem}

Some remarks related to the formulation of Theorem \ref{theorem:hmt_poincare} are in order. In \cite{hofmannmitreataylor}, Theorem \ref{theorem:hmt_poincare} is formulated assuming only that $\pom$ is Ahlfors--David regular and $\Omega$ satisfies the $2$-sided local $(D_0,R_0)$-John condition (recall Definition \ref{defin:local_local_john}), and the result holds for all $r \in (0,R_0)$ instead of all $r \in (0,\diam(\pom))$. As we noted in Example \ref{example:local_local_john}, if $R_0$ is not large enough, the $2$-sided local $(D_0,R_0)$-John condition is not strong enough to imply that $\pom$ is connected. Since Heinonen--Koskela type weak Poincar\'e inequalities (recall Definition \ref{defin:weak_poincare}) imply connectivity (see e.g. \cite[Theorem 17.1]{cheeger}), we know that the Hofmann--Mitrea--Taylor Poincar\'e inequalities may hold even when Heinonen--Koskela type weak Poincar\'e inequalities fail.

This being said, if $\Omega$ is a $2$-sided chord-arc domain (as we assumed in the beginning of this section), then the Hofmann--Mitrea--Taylor Poincar\'e inequalities hold as in Theorem \ref{theorem:hmt_poincare} and they self-improve to Heinonen--Koskela type Poincar\'e inequalities. The first claim is due to \cite[Lemma 3.13]{hofmannmitreataylor}: any NTA domain satisfies the local $(D_0,R_0)$-John condition where $R_0$ is the upper bound for the scales of the corkscrew conditions. Since we assume that the corkscrew conditions hold for any $0 < r < \diam(\pom)$, we get the local $D_0$-John condition. The second claim follows from Corollary \ref{corollary:poincare} which we prove using a Hofmann--Mitrea--Taylor type Poincar\'e estimate (see Lemma \ref{lemma:hmt1_poincare} below) in the next section.

For the case $p > 1$, we can use directly the estimate \eqref{ineq:hmt_poincare} together with the arguments in the next section to obtain a weak $p$-Poincar\'e inequality. However, for the case $p=1$, we need to revisit estimate \eqref{ineq:hmt_poincare} since an inspection of its proof shows us that a simple limiting argument does not work. Because of this, we prove the following weak $(1,1)$ type version of the Hofmann--Mitrea--Taylor Poincar\'e inequality:

\begin{lemma}
 \label{lemma:hmt1_poincare}
  There exists  a constant $C = C(\Omega)$ such that the following holds. For every compactly supported Lipschitz function $f$ on $\pom$ and every $\xi \in \pom$, $r > 0$ and $\Delta \coloneqq \Delta(\xi,r)$, there exists $a_\Delta \in \R$ such that
  \begin{align}
    \nonumber \sup_{\lambda>0} \lambda \,\sigma\big(\big\{y\in\Delta \colon |f(y) - a_\Delta|>\lambda\big\}\big)
    & \le Cr  \int_{5\Delta} |\nablahmt f| \, d\sigma \\
    \label{ineq:hmt1_poincare} &\quad+ Cr \sum_{j=2}^{\infty} 2^{-j(n+1)} \int_{2^j \Delta \setminus 2^{j-1} \Delta} |\nablahmt f| \, d\sigma.
  \end{align}
\end{lemma}

The proof of Lemma \ref{lemma:hmt1_poincare} is similar to the proof of estimate \eqref{ineq:hmt_poincare} which uses heavily the machinery built in \cite{hofmannmitreataylor}. For the convenience of the reader, we provide the key arguments below. For the proof, we recall the definitions of the Riesz transform $\Rc_\sigma$, the maximal truncation of the Riesz transform $\Rc_{\sigma,*}$ and the double layer potential $\Dc$ of a suitable function $f$ on $\pom$: for $\eps > 0$, $X \in \R^{n+1}$ and $Z \in \Omega$, we set
\begin{align*}
  \Rc_{\sigma,\eps} f(X) &\coloneqq \frac{1}{C_n} \int_{\pom \setminus B(X,\eps)} \frac{X-y}{|X-y|^{n+1}} f(y) \, d\sigma(y), \\
  \Rc_\sigma f(X) &\coloneqq \lim_{\eps \to 0} \Rc_{\sigma,\eps} f(X), \\
  \Rc_{\sigma,*} f(X) &\coloneqq \sup_{\eps > 0} |\Rc_{\sigma,\eps} f(X)|, \\
  \Dc f(Z) &\coloneqq \frac{1}{C_n} \int_{\pom} \frac{\nu(y) \cdot (y - Z)}{|Z - y|^{n+1}} f(y) \, d\sigma(y)
\end{align*}
where $\nu$ is the measure-theoretic outer unit normal of $\Omega$ (see (2.2.11) in \cite{hofmannmitreataylor}) and $C_n$ is the surface area of the unit sphere in $\R^{n+1}$. We extend $\Dc f$ to the whole $\R^{n+1} \setminus \pom$ by changing the direction of the normal $\nu$ for points $X \in \text{int} \, \Omega^{\text{c}}$.

\begin{proof}
  Let $f$ be a compactly supported Lipschitz function on $\pom$, $\xi \in \pom$ and $0 < r < \diam(\pom)$. Let $\Delta \coloneqq \Delta(\xi,r) = B(\xi,r) \cap \pom$. By the theory of layer potentials we know that for $\sigma$-a.e.\ $x \in \pom$, we have
  \begin{align*}
    f(x) = \Dc^+ f(x) - \Dc^-f(x),
  \end{align*}
  where $\Dc^+ f(x)$ and $\Dc^-f(x)$ are the inner and outer non-tangential limits of $\Dc f$ at $x$, respectively (see e.g. \cite[Section 3.3]{hofmannmitreataylor} and apply the results separately for $\Omega$ and $\text{int} \, \Omega^{\text{c}}$). These non-tangential limits exist $\sigma$-almost everywhere. For a constant vector $h_\Delta$ to be fixed below, we consider the function
  \begin{align*}
    u(X) \coloneqq \Dc f(X) - X\cdot h_\Delta,\quad X\in\R^{n+1}\setminus\pom.
  \end{align*}
  Then, by the $\sigma$-a.e. existence of the limits $\Dc^{\pm} f$, the inner and outer non-tangential limits $u^\pm(x)$ of $u$ exist for $\sigma$-a.e.\ $x\in\pom$. Thus, we have
  \begin{align*}
    f(x) = u^+(x) - u^-(x) \quad\mbox{ for $\sigma$-a.e.\ $x\in\pom$.}
  \end{align*}

  To prove the estimate \eqref{ineq:hmt1_poincare}, we choose
  \begin{align*}
    a_\Delta = u(X_{\Delta}^+) - u(X_{\Delta}^-),
  \end{align*}
  where $X_{\Delta}^+$ and $X_{\Delta}^-$ are interior and exterior local John points inside $B(\xi,r)$, respectively. Since $B(X_\Delta^{\pm},cr) \subset B(\xi,r) \setminus \pom$, we know that $\delta(X_\Delta^\pm) \approx r$. Now, for any $\lambda>0$, we have
  \begin{align}
    \nonumber \sigma\big(\big\{x\in\Delta \colon |f(x) - a_\Delta|>\lambda\big\}\big) & \leq 
    \sigma\big(\big\{x\in\Delta \colon |u^+(x) - u(X_{\Delta}^+)|>\lambda/2\big\}\big) \\
    \label{ineq:level_set} &\quad + 
    \sigma\big(\big\{x\in\Delta \colon |u^-(x) - u(X_{\Delta}^-)|>\lambda/2\big\}\big).
  \end{align}
  We only estimate the first term on the right hand side of \eqref{ineq:level_set} since the second one is estimated similarly. For any $x \in \Delta$, let $\gamma_x^+$ be a non-tangential path in $\Omega$ from $x$ to $X_{\Delta}^+$. Such a non-tangential path with a uniform constant exists for every $x \in \Delta$ by the local John condition. Since $\gamma_x^+$ is a non-tangential path with a uniform constant, we know that $\gamma_x^+ \subset B(\xi,Ar)$ for some fixed $A \ge 1$. Then, for any $Y\in \gamma_x^+\cap\Omega$, the mean value theorem gives us
  \begin{align*}
    |u(Y) - u(X_\Delta^+)|\leq \Hc^1(\gamma_x^+)\,\sup_{Z\in \gamma_x^+ \cap \Omega}|\nabla u(Z)| \lesssim r\,N_*(\chi_{B(\xi,Ar)}|\nabla u|)(x),
  \end{align*}
  where $N_*$ is the non-tangential maximal operator for suitable aperture constant $\alpha > 1$ as defined in \eqref{defin:non-tangential_max} (recall that a $D$-non-tangential path from $x \in \pom$ to $X \in \Omega$ travels inside the cone $\Gamma_D(x)$, like we discussed in the beginning of Section \ref{section:john_conditions}). Letting $Y \to x$ then gives us
  \begin{align*}
    |u^+(x) - u(X_\Delta^+)| \lesssim r\,N_*(\chi_{B(\xi,Ar)}|\nabla u|)(x).
  \end{align*}
  Thus,
  \begin{align}
    \label{estimate:level_set2} \sigma\big(\big\{x\in\Delta \colon |u^+(x) - u(X_{\Delta}^+)|>\lambda/2\big\}\big)
    \lesssim \sigma\big(\big\{x\in\Delta \colon r\,N_*(\chi_{B(\xi,Ar)}|\nabla u|)(x)>c\lambda\big\}\big).
  \end{align}
  
  Let us then estimate the right hand side of \eqref{estimate:level_set2}. We notice that for all $X \in \Gamma_\alpha(x) \cap B(\xi,Ar)$ it holds that $\nabla u(X) = \nabla \Dc f(X) - h_\Delta$. By (3.6.29) in \cite{hofmannmitreataylor}, we know that for all $X \in \Omega$ it holds that
  \begin{align*}
    \nabla \Dc f(X) = \bigg( \sum_i \int_{\pom} \partial_i \Ec(X-y)\,\partial_{\smallt,j,i} f(y)\,d\sigma(y) \bigg)_{1\leq j \leq n+1},
  \end{align*}
  where $\Ec$ is the fundamental solution to the Laplacian in $\R^{n+1}$, that is
  \begin{align*}
    \Ec(X) \coloneqq \left\{ \begin{array}{cl}
                               \dfrac{1}{C_n (1-n)} \dfrac{1}{|X|^{n-1}}, &\text{ if } n \ge 2, \\
                               \dfrac{1}{2\pi} \log|X|, &\text{ if } n = 1,
                             \end{array} \right. 
  \end{align*}
  for $X \in \R^{n+1} \setminus \{0\}$, where $C_n$ is the surface area of the unit sphere in $\R^{n+1}$. Thus, choosing
  \begin{align*}
    h_\Delta \coloneqq \bigg( \sum_i \int_{\pom\setminus 2A\Delta} \partial_i\Ec(\xi-y)\,\partial_{\smallt,j,i} f(y)\,d\sigma(y) \bigg)_{1\leq j \leq n+1}
  \end{align*}
  gives us
  \begin{align}
    \nonumber \nabla u(X) &= \bigg(\sum_i \int_{2A\Delta} \partial_i\Ec(X-y)\,\partial_{\smallt,j,i} f(y)\,d\sigma(y)\bigg)_{1\leq j \leq n+1}\\
    \nonumber & \quad + \bigg(\sum_i \int_{\pom\setminus2A\Delta} \big(\partial_i\Ec(X-y)- \partial_i\Ec(\xi-y)\big)\,\partial_{\smallt,j,i} f(y)\,d\sigma(y)\bigg)_{1\leq j \leq n+1}\\
    \nonumber & = \bigg(\sum_i \Rc_{i,\sigma}(\chi_{2A\Delta}\,\partial_{\smallt,j,i} f)(X)\bigg)_{1\leq j \leq n+1} \\
    \nonumber &\quad + \bigg(\sum_i \Rc_{i,\sigma}(\chi_{(2A\Delta)^{\text{c}}}\,\partial_{\smallt,j,i} f)(X)- \Rc_{i,\sigma}( \chi_{(2A\Delta)^{\text{c}}}\,\partial_{\smallt,j,i}f)(\xi)\bigg)_{1\leq j \leq n+1}\\
    \label{eq:i_and_ii} &\eqqcolon I(X) + II(X),
  \end{align}
  where $\Rc_{i,\sigma}$ stands for the $i$-th component of the Riesz transform $\Rc_\sigma$. We then have
  \begin{align*}
    \sigma\big(\big\{x\in\Delta \colon r\,N_*(\chi_{B(\xi,Ar)}|\nabla u|)(x)>c\lambda\big\}\big) &\leq
    \sigma\big(\big\{x\in\Delta \colon r\,N_*(\chi_{B(\xi,Ar)}|I|)(x)>c\lambda/2\big\}\big) \\
    &\quad + \sigma\big(\big\{x\in\Delta \colon r\,N_*(\chi_{B(\xi,Ar)}|II|)(x)>c\lambda/2\big\}\big).
  \end{align*}
  
  We first estimate the term $\sigma(\{x\in\Delta \colon r\,N_*(\chi_{B(\xi,Ar)}|I|)(x)>c\lambda/2\})$. Let $x \in \Delta$ and $Y \in \Gamma_\alpha(x) \cap B(\xi,Ar)$, where $\Gamma_\alpha(x)$ is the cone at $x$ with aperture $\alpha$ (recall \eqref{defin:cone}). Let $\eps = \eps_{x,Y} \coloneqq 2|x-Y|$ and $\Delta_\eps \coloneqq \Delta(x,\eps)$. We then have
  \begin{align}
    |I(Y)| &\le \sum_{i,j} |\Rc_{i,\sigma}(\chi_{2A\Delta}\,\partial_{\smallt,j,i} f)(Y)|\notag \\
    &\le \sum_{i,j} |\Rc_{i,\sigma}(\chi_{2A\Delta\cap \Delta_\eps}\,\partial_{\smallt,j,i} f)(Y)|\notag \\
    &\quad \quad + \sum_{i,j} |\Rc_{i,\sigma}(\chi_{2A\Delta\setminus \Delta_\eps}\,\partial_{\smallt,j,i} f)(Y) - \Rc_{i,\sigma}(\chi_{2A\Delta\setminus \Delta_\eps}\,\partial_{\smallt,j,i} f)(x)|\notag\\
    &\quad \quad
    + \sum_{i,j} |\Rc_{i,\sigma}(\chi_{2A\Delta\setminus \Delta_\eps}\,\partial_{\smallt,j,i} f)(x)|.\label{ineq:spp}
  \end{align}
  Since $\dist(Y,\pom)\approx\varepsilon$, the first term on the right hand side of \eqref{ineq:spp} satisfies
  \begin{align*}
    \sum_{i,j} |\Rc_{i,\sigma}(\chi_{2A\Delta\cap \Delta_\eps}\,\partial_{\smallt,j,i} f)(Y)|
    \lesssim \sum_{i,j} \int_{\Delta_\eps} \frac{|\chi_{2A\Delta}\,\partial_{\smallt,j,i} f|}{\varepsilon^n}\,d\sigma
    &\approx \sum_{i,j} \fint_{\Delta_\eps} |\chi_{2A\Delta}\,\partial_{\smallt,j,i} f| \,d\sigma \\
    &\le M(\chi_{2A\Delta} |\nabla_{\HMT} f|)(x),
  \end{align*}
  where we used Ahlfors--David regularity for the estimate $\eps^n \approx \sigma(\Delta(x,\eps))$ and $M$ stands for the centered maximal Hardy--Littlewood operator on $\pom$. For the second term on the right hand side of \eqref{ineq:spp}, we get
  \begin{align*}
    \sum_{i,j} |&\Rc_{i,\sigma}(\chi_{2A\Delta\setminus \Delta_\eps}\,\partial_{\smallt,j,i} f)(Y) - \Rc_{i,\sigma}(\chi_{2A\Delta\setminus \Delta_\eps}\,\partial_{\smallt,j,i} f)(x)|\\
    &\overset{\text{(A)}}{\lesssim} \int_{2A\Delta\setminus \Delta_\eps} \frac {|Y-x|}{|x-y|^{n+1}}\,|\nablahmt f(y)|\,d\sigma(y) \\
    &\lesssim \int_{\pom \setminus \Delta_\eps} \frac{\eps}{|x-y|^{n+1}} \, |\chi_{2A\Delta}(y) \nablahmt f(y)| \, d\sigma(y) \\
    &\overset{\text{(B)}}{\lesssim} \sum_{k=1}^\infty \int_{2^k \Delta_\eps \setminus 2^{k-1} \Delta_\eps} \frac{\eps}{(2^k \eps)^{n+1}} \, |\chi_{2A\Delta}(y) \nablahmt f(y)| \, d\sigma(y) \\
    &\overset{\text{(C)}}{\lesssim} \sum_{k=1}^{\infty} \frac{2^{-k}}{\sigma(2^k \Delta_\eps)} \int_{2^k \Delta_\eps \setminus 2^{k-1} \Delta_\eps} |\chi_{2A\Delta}\nablahmt f| \, d\sigma \\
    &\lesssim M(\chi_{2A\Delta}\nabla_{\HMT} f)(x),
  \end{align*}
  where we used
  \begin{enumerate}
    \item[(A)] the mean value theorem for the Riesz kernel functions $X = (X_1, X_2, \ldots, X_{n+1}) \mapsto \tfrac{X_i}{|X|^{n+1}}$ and the estimate $|\nabla \tfrac{X_i}{|X|^{n+1}}| \lesssim \tfrac{1}{|X|^{n+1}}$,
    
    \item[(B)] the fact that $x$ is the center point of $\Delta_\eps$ and hence, $|x-y| \approx 2^k \eps$ for $y \in 2^k \Delta_\eps \setminus 2^{k-1} \Delta_\eps$,
    
    \item[(C)] Ahlfors--David regularity.
  \end{enumerate}
  In addition, the last term on the right hand side of \eqref{ineq:spp} is bounded above by the sum $\sum_{i,j} \Rc_{i,\sigma,*}(\chi_{2A\Delta}\,\partial_{\smallt,j,i} f)(x)$, where $\Rc_{i,\sigma,*}$ stands for the maximal truncation of the Riesz transform defined using only the $i$-th component of $\Rc_\sigma$. Thus,
  \begin{align*}
    |I(Y)| \lesssim M(\chi_{2A\Delta}\nabla_{\HMT} f)(x) + \sum_{i,j} \Rc_{i,\sigma,*}(\chi_{2A\Delta}\,\partial_{\smallt,j,i} f)(x),
  \end{align*}
  and so
  \begin{align*}
    \sigma\big(\big\{x\in\Delta \colon r\,N_*(\chi_{B(\xi,2r)}|I|)(x)>c\lambda/2\big\}\big)
    & \leq \sigma\big(\big\{x\in\Delta \colon M(\chi_{2A\Delta}\nabla_{\HMT} f)(x)>c'\lambda/r\big\}\big)\\
    & \!\! +\!
    \sum_{i,j} \sigma\big(\big\{x\in\Delta \colon \Rc_{i,\sigma,*}(\chi_{2A\Delta}\,\partial_{\smallt,j,i} f)(x)> c'\lambda/r\big\}\big).
  \end{align*}
  Since $\pom$ is uniformly rectifiable, $\Rc_\sigma$ is bounded from $L^2(\sigma)$ to $L^2(\sigma)$ \cite{davidsemmes_singular}, and therefore $\Rc_{i,\sigma,*}$ is of weak type $(1,1)$ with respect to $\sigma$ by classical Calder\'on--Zygmund type techniques (see e.g. \cite[Section 5]{grafakos}). This and the weak type $(1,1)$ of the Hardy--Littlewood maximal operator combined with the previous estimates then give us
  \begin{align}
    \label{eq:aki7} \sigma\big(\big\{x\in\Delta:\,r\,N_*(\chi_{B(\xi,Ar)}|I|)(x)>c\lambda/2\big\}\big)\lesssim \frac r\lambda \int_{2A\Delta} |\nablahmt f|\,d\sigma.
  \end{align}

  Let us then consider the term $II(X)$ in \eqref{eq:i_and_ii}. For $x \in \Delta$ and $X \in \Gamma_\alpha(x) \cap B(\xi,Ar)$, the same arguments as with the middle sum in \eqref{ineq:spp} give us
  \begin{align*}
    |II(X)| \lesssim \int_{\pom\setminus2A\Delta} \frac r{|\xi-y|^{n+1}}\,|\nablahmt f(y)|\,d\sigma(y)\lesssim
    \sum_{j=2}^{\infty} \frac{2^{-j}}{\sigma(2^j\Delta)} \int_{2^j \Delta \setminus 2^{j-1} \Delta} |\nablahmt f| \, d\sigma.
  \end{align*}
  Therefore,
  \begin{align}
    \label{eq:aki8} \sigma\big(\big\{x\in\Delta \colon r\,N_*(\chi_{B(\xi,2r)}|II|)(x)>c\lambda/2\big\}\big)
    \lesssim  \frac r\lambda\,\sigma(\Delta) \sum_{j=2}^{\infty} 2^{-j(n+1)} \int_{2^j \Delta \setminus 2^{j-1} \Delta} |\nablahmt f| \, d\sigma.
  \end{align}
  Combining \eqref{eq:aki7} and \eqref{eq:aki8} with \eqref{estimate:level_set2} and other previous estimates gives us
  \begin{align*}
    \sigma\big(\big\{x\in\Delta:\,|u^+(x) - u(X_{\Delta}^+)|>\lambda/2\big\}\big)
    & \lesssim \frac r\lambda  \int_{2A\Delta} |\nablahmt f| \, d\sigma \\
    &\quad+ \frac r\lambda \sum_{j=2}^{\infty} 2^{-j(n+1)} \int_{2^j \Delta \setminus 2^{j-1} \Delta} |\nablahmt f| \, d\sigma.
  \end{align*}
  By a straightforward covering argument, the right hand side of the preceding inequality is comparable to the right hand side of \eqref{ineq:hmt1_poincare}, with the comparability constant depending on $A$. By similar arguments, one can check that the same estimate holds replacing $u^+$ by $u^-$ and $X_{\Delta}^+$ by $X_{\Delta}^-$. The claim follows then by applying the inequalities for $u^+$ and $u^-$ to \eqref{ineq:level_set}.
\end{proof}

For the proof of the case $\diam(\pom) < \infty$ in Theorem \ref{theorem:boundary_poincare}, we give a short proof of the fact that Theorem \ref{theorem:hmt_poincare} implies quasiconvexity for bounded $2$-sided chord-arc domains:

\begin{lemma}
  \label{lemma:boundary_quasiconvex}
  Suppose that $\Omega$ is a $2$-sided chord-arc domain such that $\diam(\pom) < \infty$. Then the boundary $\pom$ is quasiconvex.
\end{lemma}

Notice that Corollary \ref{corollary:poincare} implies stronger connectivity properties than the conclusion of Lemma \ref{lemma:boundary_quasiconvex} (recall Corollary \ref{corollary:annular_quasiconvexity}) but we need Lemma \ref{lemma:boundary_quasiconvex} to prove Corollary \ref{corollary:poincare}. Lemma \ref{lemma:boundary_quasiconvex} follows almost directly from the results reviewed in this section when we combine them with the following result of Durand-Cartagena, Jaramillo and Shanmugalingam:

\begin{theorem}[{\cite[Theorem 3.6]{durandcartagenaetal}}]
  \label{theorem:poincare_quasiconvex}
  Let $(X,d,\mu)$ be a complete metric measure space with a doubling measure $\mu$. Suppose that for every bounded Lipschitz function $f$ that is locally $1$-Lipschitz there exists a functional $a_f \colon \Bc \to [0,\infty)$ such that
  \begin{align*}
    \fint_B |f - \langle f \rangle_B| \, d\mu \le a_f(B) \le C r_B,
  \end{align*}
  where $\Bc$ is the collection of all open balls in $(X,d)$, $C$ is a uniform constant and $r_B$ is the radius of the ball $B$. Then the space $(X,d)$ is quasiconvex.
\end{theorem}

\begin{proof}[Proof of Lemma \ref{lemma:boundary_quasiconvex}]
  Let $f$ be a bounded Lipschitz function on $\pom$ such that $f$ is locally $1$-Lipschitz, and let $\Bc = \{\Delta(x,r) \colon x \in \pom, r < \diam(\pom)\}$. Since $\diam(\pom) < \infty$, we know that $\nablahmt f$ exists (recall Remark \ref{remark:compact_support}). Let us define the functional $a_f \colon \Bc \to \R$ by setting
  \begin{align*}
    a_f(\Delta) = C_2 r \left( \fint_{\Delta} |\nablahmt f|^2 \, d\sigma \right)^{1/2}
    + C_2 r \sum_{j=2}^{\infty} \frac{2^{-j}}{\sigma(2^j \Delta)} \int_{2^j \Delta \setminus 2^{j-1} \Delta} |\nablahmt f| \, d\sigma
  \end{align*}
  for every $\Delta = \Delta(x,r) \in \Bc$, where $C_2$ is the constant given by Theorem \ref{theorem:hmt_poincare} for the arbitrary choice $p = 2$. By part (1) of Lemma \ref{lemma:properties_of_hmt_objects}, the functional $a_f$ is well-defined, and the fact that $a_f(\Delta) < \infty$ for each $\Delta \in \Bc$ follows from the argument below. Now, by Lemma \ref{lemma:properties_of_hmt_objects}, Lemma \ref{lemma:norm_of_tangential_gradient}, and the fact that $f$ is locally $1$-Lipschitz, we know that $|\nablahmt f| \le 1$ almost everywhere. In particular, for $\Delta = \Delta(x,r) \in \Bc$ we have
  \begin{align}
    \label{estimate:functional} a_f(\Delta)
    \le C_2 r + C_2 r \sum_{j=2}^{\infty} \frac{2^{-j}}{\sigma(2^j \Delta)} \sigma(2^j \Delta \setminus 2^{j-1} \Delta)
    \le C_2 r + \sum_{j=2}^{\infty} 2^{-j}
    \le 2 C_2 r.
  \end{align}
  Thus, we now only need to notice that by H\"older's inequality and the estimate \eqref{estimate:functional}, we have
  \begin{align*}
    \fint_\Delta |f - \langle f \rangle_\Delta| \, d\sigma
    \le \left( \fint_\Delta |f - \langle f \rangle_\Delta|^2 \, d\sigma \right)^{1/2}
    \le a_f(\Delta)
    \le 2 C_2 r
  \end{align*}
  for any $\Delta = \Delta(x,r) \in \Bc$, and the claim follows from Theorem \ref{theorem:poincare_quasiconvex}.
\end{proof}

\section{Weak $1$-Poincar\'e inequality for boundaries of $2$-sided chord-arc domains}
\label{section:poincare}

Let $\Omega$ be a $2$-sided chord-arc domain. In this section, we prove Theorem \ref{theorem:boundary_poincare} and Corollary \ref{corollary:poincare} with the help of some tools from the literature. As a simple consequence of Theorem \ref{theorem:boundary_poincare} and some results in the literature, we also show that the tail in the Hofmann--Mitrea--Taylor weak Poincar\'e inequality (Theorem \ref{theorem:hmt_poincare}) can be removed, at least when $\Omega$ is a bounded $2$-sided chord-arc domain (see Corollary \ref{corollary:poincare_for_hmt_gradient}).

Instead of proving a Poincar\'e type inequality directly, we use the following result to reduce the proof to a pointwise estimate:

\begin{theorem}[{\cite[part of Theorem 8.1.7]{heinonenetal}}]
  \label{theorem:poincare_pointwise}
  Let $(X,d,\mu)$ be a metric measure space with a doubling measure $\mu$ and $V$ be a Banach space. Suppose that $1 \le p < \infty$, $u \colon X \to V$ is integrable on balls and $g \colon X \to [0,\infty]$ is measurable. Then the following two conditions equivalent:
  \begin{enumerate}
    \item[(a)] there exist constants $C, \lambda \ge 1$ such that
    \begin{align*}
      \fint_B \left| u(x) - \langle u \rangle_B \right| \, d\mu(x) \le C \diam(B) \left( \fint_{\lambda B} g(x)^p \, d\mu(x) \right)^{1/p}
    \end{align*}
    for every open ball $B$ in $X$,
    
    \item[(b)] there exist constants $C, \lambda \ge 1$ such that
    \begin{align*}
      |u(x) - u(y)| \le C d(x,y) \left( M_{\lambda d(x,y)}(g^p)(x) + M_{\lambda d(x,y)} (g^p)(y) \right)^{1/p}
    \end{align*}
    for almost all $x,y \in X$, where $M_R$ is the $R$-truncated centered Hardy--Littlewood maximal operator on $\pom$, $M_R f(z_0) \coloneqq \sup_{r < R} \fint_{\Delta(z_0,r)} |f(z)| \, d\sigma(z)$.
  \end{enumerate}
\end{theorem}
These types of characterizations with respect to pointwise inequalities originate from \cite{heinonenkoskela}.

Thus, to prove Theorem \ref{theorem:boundary_poincare}, it is enough for us to prove the following lemma (recall also Remark \ref{remark:compact_support}):

\begin{lemma}
  \label{lemma:pointwise_bound}
  Suppose that $u$ is a compactly supported Lipschitz function on $\pom$. There exists a universal constant $C \ge 1$ such that
  \begin{align*}
    |u(x) - u(y)| \le C |x-y| \left( M_{C|x-y|}(|\nablat u|)(x) + M_{C|x-y|}(|\nablat u|)(y) \right)
  \end{align*}
  for all points $x,y \in \pom$ for which the tangential gradient $\nablat u$ exists and \eqref{bound:norm_of_tangential_gradient} holds.
\end{lemma}

Let $u$ be a compactly supported Lipschitz function on $\pom$. As we noted in Section \ref{subsection:gradients}, the tangential gradient $\nablat u$ exists for almost every point $x \in \pom$. By Lemma \ref{lemma:norm_of_tangential_gradient}, we know that \eqref{bound:norm_of_tangential_gradient} holds for almost every point $x \in \pom$. Thus, the points in Lemma \ref{lemma:pointwise_bound} are almost all points in $\pom$, as is required in Theorem \ref{theorem:poincare_pointwise}.

In the proof of Lemma \ref{lemma:pointwise_bound}, we use a smooth cutoff function for balls. The construction of the function uses the usual mollifier technique. For the convenience of the reader -- particularly because we need a quantitative bound for the norm of the gradient -- we give the key details below.

Let us start by defining the standard mollifier. We set $\eta \colon \R^{n+1} \to \R$,
\begin{align*}
  \eta(X) = \left\{ \begin{array}{cl}
                      c_n e^{-1 / (1 - |X|^2)}, &\text{if } |X| < 1 \\
                      0, &\text{if } |X| \ge 1
                    \end{array} \right. ,
\end{align*}
where the constant $c_n$ is chosen so that $\int_{\R^{n+1}} \eta(X) \, dX = 1$. For any $\kappa > 0$, we set
\begin{align*}
  \eta_\kappa(X) = \frac{1}{\kappa^{n+1}} \eta\left(\frac{X}{\kappa}\right).
\end{align*}
Notice that $\text{supp} \, \eta_\kappa \subset B(0,\kappa)$. Using the standard mollifier, we define the smooth cutoff function $\varphi_\kappa$ for the ball $B(0,\kappa)$ characteristic function using convolutions: we set $\varphi_\kappa \colon \R^{n+1} \to \R$,
\begin{align}
  \label{defin:cutoff_function} \varphi_\kappa(X) = \eta_{\kappa} * \chi_{B(0,\kappa)}(X).
\end{align}
Thus, we have
\begin{align*}
  \varphi_\kappa(X) 
  &= \int_{\R^{n+1}} \eta_{\kappa}(X-Y) \, \chi_{B(0,\kappa)}(Y) \, dY \\
  &= \int_{\R^{n+1}} \eta_{\kappa}(Y) \, \chi_{B(0,\kappa)}(X-Y) \, dY = \int_{B(0,\kappa)} \eta_{\kappa}(Y) \, \chi_{B(0,\kappa)}(X-Y) \, dY.
\end{align*}
From this representation, we see that $\varphi_\kappa \equiv 1$ on $B(0,\kappa)$ and $\varphi_\kappa \equiv 0$ on $\R^{n+1} \setminus B(0,2\kappa)$. By the standard theory of mollifiers (see e.g. \cite[p.~123--124]{evansgariepy}), we know that $\varphi_\kappa$ is smooth and it satisfies
\begin{align}
  \label{gradient_integral} \nabla \varphi_\kappa(X) = \int_{\R^{n+1}} \nabla_X \eta_{\kappa}(X-Y) \, \chi_{B(0,\kappa)}(Y) \, dY.
\end{align}
In particular, by the construction, the smoothness and compact support of $\eta$ and \eqref{gradient_integral}, for $X = (X_1, X_2, \ldots, X_{n+1})$ we have
\begin{align*}
  \left| \frac{\partial_i}{\partial X_i} \varphi_\kappa(X) \right| 
  &= \left| \int_{\R^{n+1}} \frac{\partial_i}{\partial X_i} \, \eta_{\kappa}(X-Y) \chi_{B(0,\kappa)}(Y) \, dY \right| \\
  &\le \int_{\R^{n+1}} \left| \frac{\partial_i}{\partial X_i} \eta_{\kappa}(X-Y) \right| \, dY = \frac{1}{\kappa} \int_{\R^{n+1}} \left| \frac{\partial_i}{\partial X_i} \eta(X-Y) \right| \, dY \le \frac{C}{\kappa},
\end{align*}
where the constant $C$ does not depend $\kappa$. Thus, we have $|\nabla \varphi_\kappa(X)| \lesssim \tfrac{1}{\kappa}$. We note that this implies that $\varphi_\kappa$ is a compactly supported Lipschitz function. Therefore the tangential gradient of $\varphi_\kappa$ exists for almost every point $x \in \pom$, and \eqref{estimate:gradient_norm} and \eqref{lemma:norm_of_tangential_gradient} imply that
\begin{align}
  \label{estimate:gradient_norm} |\nablat \varphi_\kappa(x)| \lesssim \frac{1}{\kappa}.
\end{align}
By translating and adjusting the constant $\kappa$, we can construct a smooth cutoff function for any ball $B(X,r)$ in $\R^{n+1}$.

\begin{proof}[Proof of Lemma \ref{lemma:pointwise_bound}]
  Let us fix points $x,y \in \pom$ such that the tangential gradient exists at $x$ and $y$ and \eqref{bound:norm_of_tangential_gradient} holds for $x$ and $y$.
  
  We prove the bound by using truncation  and localization arguments which help us to control the values of the function inside balls $\Delta(x,c|x-y|)$ and $\Delta(y,c|x-y|)$ and allow us to deal with the tail in the right hand side of inequality \eqref{ineq:hmt1_poincare}. 
   We first assume that $\diam(\pom) = \infty$. Without loss of generality, we may assume that $u(x) = 0 < u(y)$. In particular, we have $|u(x) - u(y)| = u(y)$. We set
  \begin{align*}
    \widetilde{u}(z) \coloneqq \left\{ \begin{array}{cl}
                                         u(z), &\text{if } \ u(x) < u(z) < u(y) \\
                                         0,    &\text{if } \ u(z) \le 0 \\
                                         u(y), &\text{if } \ u(z) \ge u(y)
                                       \end{array} \right. .
  \end{align*}
  Notice that $\widetilde{u}$ is a Lipschitz function since we get it by truncating the Lipschitz function $u$ from above and below. Thus, by Lemma \ref{lemma:norm_of_tangential_gradient}, we have
  \begin{align*}
    |\nablat u(z_0)| = \limsup_{\pom \ni z \to z_0} \frac{|u(z_0) - u(z)|}{|z_0-z|}
    \ \ \ \text{and} \ \ \ 
    |\nablat \widetilde{u}(z_0)| = \limsup_{\pom \ni z \to z_0} \frac{|\widetilde{u}(z_0) - \widetilde{u}(z)|}{|z_0-z|}
  \end{align*}
  for $\sigma$-a.e. $z_0 \in \pom$. Let $z_0 \in \pom$ be a point such that the previous holds. By continuity, we have the following:
  \begin{enumerate}
    \item[i)] If $u(z_0) < u(x) = 0$ or $u(z_0) > u(y)$, then $|\nablat \util(z_0)| = 0 \le |\nablat u(z_0)|$.
    
    \item[ii)] If $u(x) < u(z_0) < u(y)$, then $|\nablat \util(z_0)| = |\nablat u(z_0)|$.
    
    \item[iii)] If $u(z_0) = u(x) = 0$, then there exists $r = r_{z_0} > 0$ such that if $|z_0-z| < r$, then $u(z) < u(y)$. For such $z$, 
    \begin{enumerate}
      \item[$\bullet$] if $u(z) > u(x)$, then $|\util(z_0) - \util(z)| = |u(z_0) - u(z)|$, and
      \item[$\bullet$] if $u(z) \le u(x)$, then $|\util(z_0) - \util(z)| = 0 \le |u(z_0) - u(z)|$.
    \end{enumerate}
    In particular, we have $|\nablat \util(z_0)| \le |\nablat u(z_0)|$.
    
    \item[iv)] If $u(z_0) = u(y)$, then $|\nablat \util(z_0)| \le |\nablat u(z_0)|$ by a similar argument as in iii).
  \end{enumerate}
  Thus, for almost every $z_0 \in \pom$, the tangential gradients exist for $u$ and $\util$ and we have
  \begin{align}
    \label{bound:tangential_gradients} |\nablat \util(z_0)| \le |\nablat u(z_0)|.
  \end{align}
  
  Let us then start processing $|u(x) - u(y)|$. We denote
  \begin{align*}
    R \coloneqq |x-y|, \ \
    \Delta_0 \coloneqq \Delta(x,R), \ \
    \Delta_0' \coloneqq \Delta(y,R), \ \
    \Delta_j \coloneqq 2^{-j} \Delta_0, \ \ \text{and} \ \
    \Delta_j' \coloneqq 2^{-j} \Delta_0'.
  \end{align*}
  Let $\varphi = \varphi_{2^\alpha R}$ be a smooth cutoff function for the ball $B(x,2^\alpha R)$ as in \eqref{defin:cutoff_function} (after translation, for the choice $\kappa = 2^\alpha R$) for large $\alpha \in \N$ to be fixed later, and let $v \coloneqq \varphi\,\util$. By Lemma \ref{lemma:hmt1_poincare}, there exist numbers $a_{\Delta_j} \in \R$ such that
  \begin{align}\label{eq:aki1}
    \sup_{\lambda>0} \lambda \,\sigma\big(\big\{z\in\Delta_j \colon |v(z) - a_{\Delta_j}|>\lambda\big\}\big)
    \le Cr_j  \,\sigma(\Delta_j)\,S_j,
  \end{align}
  where $r_j \coloneqq 2^{-j}R$ and
  \begin{align}
   \label{defin:sj} S_j \coloneqq \fint_{5\Delta_j} |\nablahmt v | \, d\sigma + \sum_{k=2}^{\infty} \frac{2^{-k}}{\sigma(2^k\Delta_j)} \int_{2^k \Delta_j \setminus 2^{k-1} \Delta_j} |\nablahmt v| \, d\sigma.
  \end{align}
  An analogous estimate holds when we replace $\Delta_j$, $a_{\Delta_j}$ and $S_j$ by $\Delta_j'$, $a_{\Delta_j'}$ and $S_j'$, respectively, where $S_j'$ is defined as $S_j$ with $\Delta_j'$ in place of $\Delta_j$. We claim now that
  \begin{align}
    \label{eq:aki2} \lim_{j\to\infty} a_{\Delta_j}=v(x) = \util(x) = u(x)=0.
  \end{align}
  Indeed, for any $\lambda \in(0,|a_{\Delta_j}- \langle v\rangle_{\Delta_j}|)$ and any $z\in\Delta_j$, we have
  \begin{align*}
    \lambda < |a_{\Delta_j}- \langle v\rangle_{\Delta_j}| \le |v(z) - a_{\Delta_j}| + |v(z) - \langle v\rangle_{\Delta_j}|.
  \end{align*}
  Hence, by \eqref{eq:aki1} and Chebyshev's inequality, we get
  \begin{align*}
    \sigma(\Delta_j) 
    &\le \sigma\big(\big\{z\in\Delta_j:\,|v(z) - a_{\Delta_j}|>\lambda/2\big\} + 
         \sigma\big(\big\{z\in\Delta_j:\,|v(z) -  \langle v\rangle_{\Delta_j}|>\lambda/2\big\} \\
    &\lesssim \frac1\lambda\,r_j  \,\sigma(\Delta_j)\,S_j + \frac1\lambda\int_{\Delta_j}|v -  \langle v\rangle_{\Delta_j}|\,d\sigma
  \end{align*}
  Therefore,
  \begin{align}
    \label{ineq:lambda} \lambda\lesssim r_j  \,S_j + \fint_{\Delta_j}|v -  \langle v\rangle_{\Delta_j}|\,d\sigma.
  \end{align}
  Since $v$ is a compactly supported Lipschitz function, Lemma \ref{lemma:properties_of_hmt_objects} and Lemma \ref{lemma:norm_of_tangential_gradient} give us $|S_j|\lesssim\|\nablahmt v\|_{L^\infty(\sigma)}<\infty$, and thus $r_j  \,S_j\to 0$ as $j\to\infty$. By the continuity of $v$, we also know that $\fint_{\Delta_j}|v(z) -  \langle v\rangle_{\Delta_j}|\,d\sigma\to0$ as $j\to\infty$. Then, choosing $\lambda = |a_{\Delta_j}- \langle v\rangle_{\Delta_j}|/2$ gives us \eqref{eq:aki2} by \eqref{ineq:lambda} and continuity of $v$.
 
  Analogously to \eqref{eq:aki2}, we also have
  \begin{align*}
    \lim_{j\to\infty} a_{\Delta_j'}=v(y)=\util(y) = u(y).
  \end{align*}
  Thus,
  \begin{align*}
    u(y) = |u(x) - u(y)|
    &= |v(x) - v(y)| \\
    &= \left| \left( \sum_{j=0}^\infty \left( a_{\Delta_{j+1}} - a_{\Delta_j} \right) + a_{\Delta_0} \right) 
    - \left( \sum_{j=0}^\infty \left( a_{\Delta_{j+1}'} - a_{\Delta_j'} \right) + a_{\Delta_0'} \right) \right| \\
    &\le \sum_{j=0}^\infty \left| a_{\Delta_{j+1}} - a_{\Delta_j} \right|
    + \sum_{j=0}^\infty \left| a_{\Delta_{j+1}'} - a_{\Delta_j'} \right| + \left| a_{\Delta_0} - a_{\Delta_0'} \right| \\
    &\eqqcolon I + II + III.
  \end{align*}
  Let us consider the sum $I$ first. We analyze $I$ by applying \eqref{eq:aki1} again for $|a_{\Delta_j}- a_{\Delta_{j+1}}|$. For any $\lambda \in (0,|a_{\Delta_j}- a_{\Delta_{j+1}}|)$ and any $z\in\Delta_j$, we have
  \begin{align*}
    \lambda < |a_{\Delta_j}- a_{\Delta_{j+1}}| \leq |v(z) - a_{\Delta_j}| + |v(z) - a_{\Delta_{j+1}}|,
  \end{align*}
  and therefore \eqref{eq:aki1} gives us
  \begin{align*}
    \sigma(\Delta_j)
    &\leq \sigma\big(\big\{z\in\Delta_j \colon |v(z) - a_{\Delta_j}|>\lambda/2\big\} + 
    \sigma\big(\big\{z\in\Delta_{j+1} \colon |v(z) -  a_{\Delta_{j+1}}|>\lambda/2\big\}\\
    &\lesssim \frac1\lambda\,r_j  \,\sigma(\Delta_j)\,S_j + \frac1\lambda\,r_{j+1}  \,\sigma(\Delta_{j+1})\,S_{j+1}
    \lesssim \frac1\lambda\,r_j  \,\sigma(\Delta_j)\,S_j.
  \end{align*}
  Thus, $\lambda\lesssim r_j \,S_j.$ Since this holds for all $\lambda\in(0,|a_{\Delta_j}- a_{\Delta_{j+1}}|)$, we have
  \begin{align}
    \label{eq:adjj} |a_{\Delta_j} - a_{\Delta_{j+1}}|\lesssim r_j \,S_j.
  \end{align}
  This then gives us
  \begin{align*}
    I
    &\lesssim \sum_{j=0}^\infty r_j\,S_j \\
    &\overset{\text{(A)}}{\lesssim} \sum_{j=0}^\infty \left( 2^{-j}R \fint_{5 \Delta_j} |\nablat (\util\varphi)| \, d\sigma 
    + 2^{-j}R \sum_{k=2}^{j+\alpha+1} 2^{-k} \fint_{2^k \Delta_j} |\nablat (\util\varphi)| \, d\sigma \right) \\
    &\overset{\text{(B)}}{\le} \sum_{j=0}^\infty 2^{-j}R \sum_{k=2}^{j+\alpha+1} 2^{-k}  \fint_{2^k \Delta_j} |\nablat\util| \, d\sigma \\
    &\hspace{2cm} + \sum_{j=0}^\infty 2^{-j}R \cdot 2^{-j-\alpha-1}  \frac{1}{\sigma(2^{\alpha+1}\Delta_0)} \int_{2^{\alpha+1}\Delta_0 \setminus 2^\alpha \Delta_0} |\util(z) \nablat\varphi(z)|\, d\sigma(z) \\
    &\overset{\text{(C)}}{\lesssim} R \cdot M_{2^{\alpha+2}R}(\nablat \util)(x) + \frac{\util(y)}{2^\alpha},
  \end{align*}
  where we used
  \begin{enumerate}
    \item[(A)] the definition of $S_j$ (see \eqref{defin:sj}), Lemma \ref{lemma:properties_of_hmt_objects} and the fact that $\varphi$ vanishes outside $2^{\alpha+1} \Delta_0$,
    \item[(B)] Ahlfors--David regularity, the product rule for gradients and the properties of the cutoff function $\varphi$: $\varphi \le 1$ everywhere and $\varphi$ is constant on $2^{\alpha} \Delta_0$ and outside $2^{\alpha+1} \Delta_0$,
    \item[(C)] the fact that $2^k \Delta_j \subset 2^{\alpha+1} \Delta_0$ for all the relevant $k$ and $j$, the definition of the truncated Hardy--Litlewood maximal operator, the fact that $\util(z) \le \util(y)$ for all $z$, and \eqref{estimate:gradient_norm}.
  \end{enumerate}
  Using the same techniques gives us the bound
  \begin{align*}
    II \lesssim R \cdot M_{2^{\alpha+2}R}(\nablat \util)(y) + \frac{\util(y)}{2^\alpha}.
  \end{align*}  
  As for $III$, we notice that
  \begin{align*}
    III
    \leq \big| a_{\Delta_0} - a_{2\Delta_0}\big| + \big|a_{2\Delta_0} -  a_{\Delta_0'} \big| 
    \lesssim r_0\,S_0.
  \end{align*}
  Indeed, the fact that $| a_{\Delta_0} - a_{2\Delta_0}| \lesssim r_0\,S_0$ follows from \eqref{eq:adjj}, and an analogous estimate holds for the term $|a_{2\Delta_0} -  a_{\Delta_0'}|$. Thus, the estimate obtained above for the term $I$ is also valid for $III$:
  \begin{align*}
    III \lesssim \sum_{j=0}^\infty r_j\,S_j
    \lesssim R \cdot M_{2^{\alpha+2}R}(\nablat \util)(x) + \frac{\util(y)}{2^\alpha}.
  \end{align*}

  Note that none of the implicit constants for the bounds for $I$, $II$ and $III$ depends on $\alpha$. Thus, recalling that $u(x) = 0 < u(y) = \util(y)$ and $R = |x-y|$, there exists a constant $C$ such that
  \begin{align*}
    u(y) = |u(x) - u(y)| 
    &\le I + II + III \\
    &\le CR(M_{2^{\alpha+2}R} \left(\nablat \util)(x) + M_{2^{\alpha+2}R}(\nablat \util)(y) \right) + \frac{C}{2^\alpha} \util(y) \\
    &\overset{\eqref{bound:tangential_gradients}}{\le} 2C|x-y| (M_{2^{\alpha+2}R} \left(\nablat u)(x) + M_{2^{\alpha+2}R}(\nablat u)(y) \right) + \frac{C}{2^\alpha} u(y).
  \end{align*}
  By choosing large enough $\alpha$, we may absorb $\frac{C}{2^\alpha} u(y)$ to the left-hand side. This completes the proof for the case $\diam(\pom) = \infty$.
  
  Let us then assume that $\diam(\pom) < \infty$. We have to consider this case separately because in the bounded case there might not exist cutoff functions with gentle enough gradient slope. However, the proof is still based on the previous case. The proof works as  previously if $|x-y| \ll \diam(\pom)$ and thus, we may assume that $|x-y| \approx \diam(\pom)$. Now, by Lemma \ref{lemma:boundary_quasiconvex}, we know that there exists a path $\gamma_{x,y}$ from $x$ to $y$ in $\pom$ such that $\ell(\gamma_{x,y}) \le C_0 |x-y| \approx C_0 \cdot \diam(\pom)$. Using a covering argument for $\gamma_{x,y}$, we find a uniformly bounded number of points $z_0, z_1, \ldots, z_J \in \pom$ with $z_0 = x$, $z_J = y$ and $|z_j - z_{j+1}| \ll \diam(\pom)$. We get
  \begin{align*}
    |u(x) - u(y)|
    &\le \sum_{j=0}^{J-1} |u(z_j) - u(z_{j+1})| \\
    &\lesssim \sum_{j=0}^{J-1} |z_j - z_{j+1}| \left( M_{C|z_j - z_{j+1}|}(|\nablat u|)(z_j) + M_{C|z_j - z_{j+1}|}(|\nablat u|)(z_{j+1}) \right) \\
    &\lesssim |x-y| \left( M_{\widetilde{C}|x-y|}(|\nablat u|)(x) + M_{\widetilde{C}|x-y|}(|\nablat u|)(y) \right),
  \end{align*}
  which is what we wanted.
\end{proof}

Corollary \ref{corollary:poincare} follows now immediately when we combine Theorem \ref{theorem:boundary_poincare} and Lemma \ref{lemma:norm_of_tangential_gradient} with the following key result of Keith (which is an improvement of an earlier result of Heinonen and Koskela \cite[Theorem 1.1]{heinonenkoskela_note}):

\begin{theorem}[{\cite[Theorem 2]{keith}}]
  \label{theorem:keith}
  Let $(X,d,\mu)$ be a complete metric measure space with a doubling measure $\mu$ and $1 \le p < \infty$. Then the following conditions are equivalent:
  \begin{enumerate}
    \item[(a)] the space $(X,d,\mu)$ supports a weak $p$-Poincar\'e inequality,
    \item[(b)] there exist constants $C, \lambda \ge 1$ such that
    \begin{align*}
      \fint_B \left| u(x) - \langle u \rangle_B \right| \, d\mu(x) \le C \diam(B) \left( \fint_{\lambda B} \left( \Lip  u(x) \right)^p \, d\mu(x) \right)^{1/p}
    \end{align*}
    for every compactly supported Lipschitz function $u$ and every ball $B$ in $X$, where
    \begin{align*}
      \Lip u(x) \coloneqq \limsup_{\substack{y \to x \\ y \neq x}} \frac{|f(x) - f(y)|}{d(x,y)}.
    \end{align*}
  \end{enumerate}
\end{theorem}

As a consequence of Theorem \ref{theorem:boundary_poincare}, we get the following improvement of Theorem \ref{theorem:hmt_poincare} in bounded $2$-sided chord-arc domains:

\begin{corollary}
  \label{corollary:poincare_for_hmt_gradient}
  Let $\Omega \subset \R^{n+1}$ be a bounded $2$-sided chord-arc domain, and let $1 < p < \infty$. There exists a constant $C = C(\Omega,p)$ such that for every $f \in L^p_1(\pom)$, $x \in \pom$ and $r > 0$ we have
  \begin{align*}
    \left( \fint_{\Delta} |f - \langle f \rangle_{\Delta}|^p \, d\sigma \right)^{1/p} \le C r \left( \fint_{\Lambda \Delta} |\nabla_{\HMT} f |^p \, d\sigma \right)^{1/p},
  \end{align*}
  where $\Delta \coloneqq \Delta(x,r)$ and $\Lambda$ is the constant from Theorem \ref{theorem:boundary_poincare}.
\end{corollary}

It is likely that the boundedness assumption is not necessary for Corollary \ref{corollary:poincare_for_hmt_gradient} but since the density results in \cite[Section 4.3]{hofmannmitreataylor} are stated in the case where the boundary $\pom$ is compact, we only consider this case. We do not consider the case $p = 1$ since the theory of the Hofmann--Mitrea--Taylor Sobolev spaces $L^p_1(\pom)$ has been developed only for $1 < p < \infty$.

\begin{proof}[Proof of Corollary \ref{corollary:poincare_for_hmt_gradient}]
  Let $f \in L^p_1(\pom)$, $x \in \pom$, $r > 0$ and $\Delta \coloneqq \Delta(x,r)$. By \cite[Corollary 4.28]{hofmannmitreataylor}, we know that Lipschitz functions form a dense subset of $L^p_1(\pom)$ when $L^p_1(\pom)$ is equipped with the natural Sobolev-type norm (see \cite[Section 3.6]{hofmannmitreataylor} for details). In particular, since $\Omega$ is a bounded $2$-sided chord-arc domain, there exists a sequence of compactly supported Lipschitz functions $(f_k)$ such that $f_k \to f$ and $\nabla_{\HMT} f_k \to \nabla_{\HMT} f$ in $L^p(\pom)$. By Theorem \ref{theorem:boundary_poincare}, the functions $f_k$ satisfy a weak $1$-Poincar\'e inequality with respect to the tangential gradient, which implies a weak $(p,p)$-Poincar\'e inequality with respect to tangential gradient by H\"older's inequality and \cite[Corollary 9.14]{heinonenetal}. These observations together with Lemma \ref{lemma:properties_of_hmt_objects} give us
  \begin{align*}
    \left(\fint_\Delta |f_k - \langle f_k \rangle_\Delta|^p \, d\sigma \right)^{1/p} \le C r \left( \fint_{\Lambda \Delta} |\nabla_{\smallt} f_k|^p \, d\sigma \right)^{1/p} = C r \left(\fint_{\Lambda \Delta} |\nabla_{\HMT} f_k|^p \, d\sigma \right)^{1/p}.
  \end{align*}
  Letting $k \to \infty$ then gives us
  \begin{align*}
    \left( \fint_\Delta |f - \langle f \rangle_\Delta|^p \, d\sigma \right)^{1/p} \le C r \left( \fint_{\Lambda \Delta} | \nabla_{\HMT} f|^p \, d\sigma \right)^{1/p}
  \end{align*}
  by standard $L^p$ convergence arguments, which is what we wanted.
\end{proof}

\section{A counterexample and questions}
\label{section:questions}

By Corollary \ref{corollary:poincare}, we know that the boundary of any $2$-sided chord-arc domain supports weak Heinonen--Koskela type Poincar\'e inequalities. It is natural to ask the following:
\begin{enumerate}
  \item[1)] Can Corollary \ref{corollary:poincare} be reversed, i.e., if $\Omega \subset \R^{n+1}$ is an open set with $n$-dimensional Ahlfors--David regular boundary such that $\pom$ supports weak Poincar\'e inequalities, is $\Omega$ a $2$-sided chord-arc domain?
  
  \item[2)] Can the assumptions of Corollary \ref{corollary:poincare} be weakened in the obvious way, i.e., if $\Omega \subset \R^{n+1}$ is a ($1$-sided) chord-arc domain with connected boundary, does $\pom$ support weak Poincar\'e inequalities?
\end{enumerate}
In the second question, the connectivity assumption for $\pom$ is necessary since weak Poincar\'e inequalities imply connectivity (see e.g. \cite[Theorem 17.1]{cheeger}). However, the answer to both of these questions is no. For the first question, this follows from Example \ref{example:non-2-sided-cad_poincare} below. For the second question, this follows from the example constructed by Mourgoglou and the second named author in \cite[Section 10]{mourgogloutolsa2}. They construct a chord-arc domain with connected boundary such that the tangential regularity problem for the Laplacian is not solvable in $L^p$ for any $1 \le p < \infty$ (recall Definition \ref{defin:regularity_problem}). The boundary of this chord-arc domain cannot support weak Poincar\'e inequalities by \cite[Theorem 1.2]{mourgogloutolsa2}.

For our example, we need some results in the literature. In particular, we need the following result of Heinonen and Koskela about how weak Poincar\'e inequalities survive under unions. We formulate the result only in our context but we note that the result holds more generally for certain types of unions of Ahlfors--David regular metric spaces.

\begin{theorem}[{\cite[special case of Theorem 6.15]{heinonenkoskela}}]
  \label{theorem:glueing}
  Let $E_1, E_2 \subset \R^{n+1}$ be two $n$-ADR sets such that both $E_1$ and $E_2$ support a weak $p$-Poincar\'e inequality for some $1 \le p < \infty$, and $\Hc^n(E_1 \cap E_2 \cap B(x,r)) \ge c r^n$ for all $x \in E_1 \cap E_2$ and all $0 < r < \min\{\diam(E_1),\diam(E_2)\}$. Then also the union $E_1 \cup E_2$ supports a weak $p$-Poincar\'e inequality when we equip $E_1 \cup E_2$ with the metric $d$ that equals the usual Euclidean distance individually on $E_1$ and $E_2$ and
  \begin{align}
    \label{defin:new_metric} d(x,y) \coloneqq \inf_{Z \in E_1 \cap E_2} |X - Z| + |Z - Y|
  \end{align}
  for $X \in E_1 \setminus E_2$ and $Y \in E_2 \setminus E_1$.
\end{theorem}

We will also use repeatedly the fact that bi-Lipschitz mappings preserve weak Poincar\'e inequalities (see \cite[Proposition 4.16]{bjornbjorn_nonlinear} for an explicit proof in the context of more general metric spaces and inequalities):

\begin{proposition}
  \label{proposition:bilipschitz_invariance}
  Let $E_1, E_2 \subset \R^{n+1}$ be two $n$-ADR sets equipped with metrics $d_1$ and $d_2$, respectively. Suppose that there exists a bi-Lipschitz mapping $\Phi \colon (E_1,d_1) \to (E_2,d_2)$ such that $\Hc^n(A) \approx \Hc^n(\Phi(A))$ for every measurable set $A \subset E_1$ and a uniform implicit constant. If $E_1$ supports a weak $p$-Poincar\'e inequality for some $1 \le p < \infty$, then also $E_2$ supports a weak $p$-Poincar\'e inequality.
\end{proposition}

Let us then construct a disconnected non-chord-arc domain example of a set whose boundary still supports a weak $1$-Poincar\'e inequality:

\begin{example}
  \label{example:non-2-sided-cad_poincare}
  Let us consider a ''twice-pinched annulus`` in $\R^2$ which we construct the following way. We start by considering the boundary of a usual annulus $A \coloneqq B(0,4) \setminus \overline{B(0,3)}$. We remove all the points that lie on the strip $\{(x,y) \in \R^2 \colon -1 < y < 1\}$. This leaves us with four circular arcs: two inner arcs and two outer arcs. We then connect these arcs to each other with four line segments so that the inner arcs connect to outer arcs, and vice versa. This leaves us with a shape that looks like an annulus that has been pinched in two places so that the interior is no longer connected but the boundary is. We let $\Omega$ be the disconnected interior of this pinched annulus (see Figure \ref{figure:pinched_annulus}).
  
  \begin{figure}[ht]
    \includegraphics[scale=0.5]{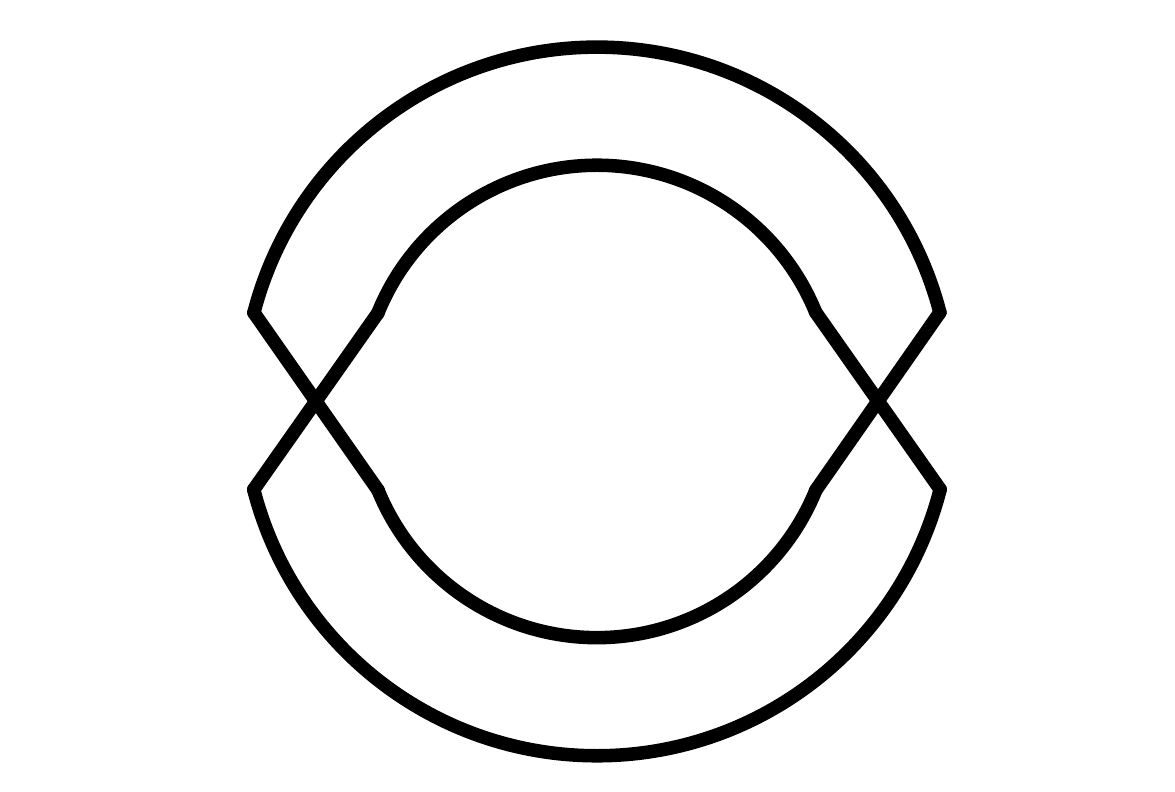}
    \caption{The set $\Omega$ in Example \ref{example:non-2-sided-cad_poincare} is the disconnected interior of a twice-pinched annulus in $\R^2$. Unlike with a usual annulus, the boundary of $\Omega$ is connected (which is one of the minimum requirements Poincar\'e inequalities).}
    \label{figure:pinched_annulus}
  \end{figure}
  
  The set $\Omega$ satisfies the $2$-sided corkscrew condition and the boundary $\pom$ is $1$-ADR but since $\Omega$ is not connected, it is not a chord-arc domain (however, we can easily modify the example to make it connected but still not a chord-arc domain; see Remark \ref{remark:pinched_annulus}). The boundary of $\Omega$ still supports a weak $1$-Poincar\'e inequality. We see this by noticing that we can express $\pom$ as a union of pieces that satisfy weak $1$-Poincar\'e inequalities and that we can glue together with ample intersections to give back $\pom$ (that is, we can use Theorem \ref{theorem:glueing} for these pieces). Indeed, $\pom$ consists of a slightly distorted inner circle and a slightly distorted outer circle. These distorted circles intersect only in two points (the two places where the line segments cross over each other) and therefore they alone are not enough for Theorem \ref{theorem:glueing}. Because of this, as two additional pieces, we take the cross-like unions of the pairs of line segments that cross over each other. These pieces have ample intersections with both of the distorted circles (see Figure \ref{figure:annulus_decomposition}).
  
  \begin{figure}[ht]
    \includegraphics[scale=0.4]{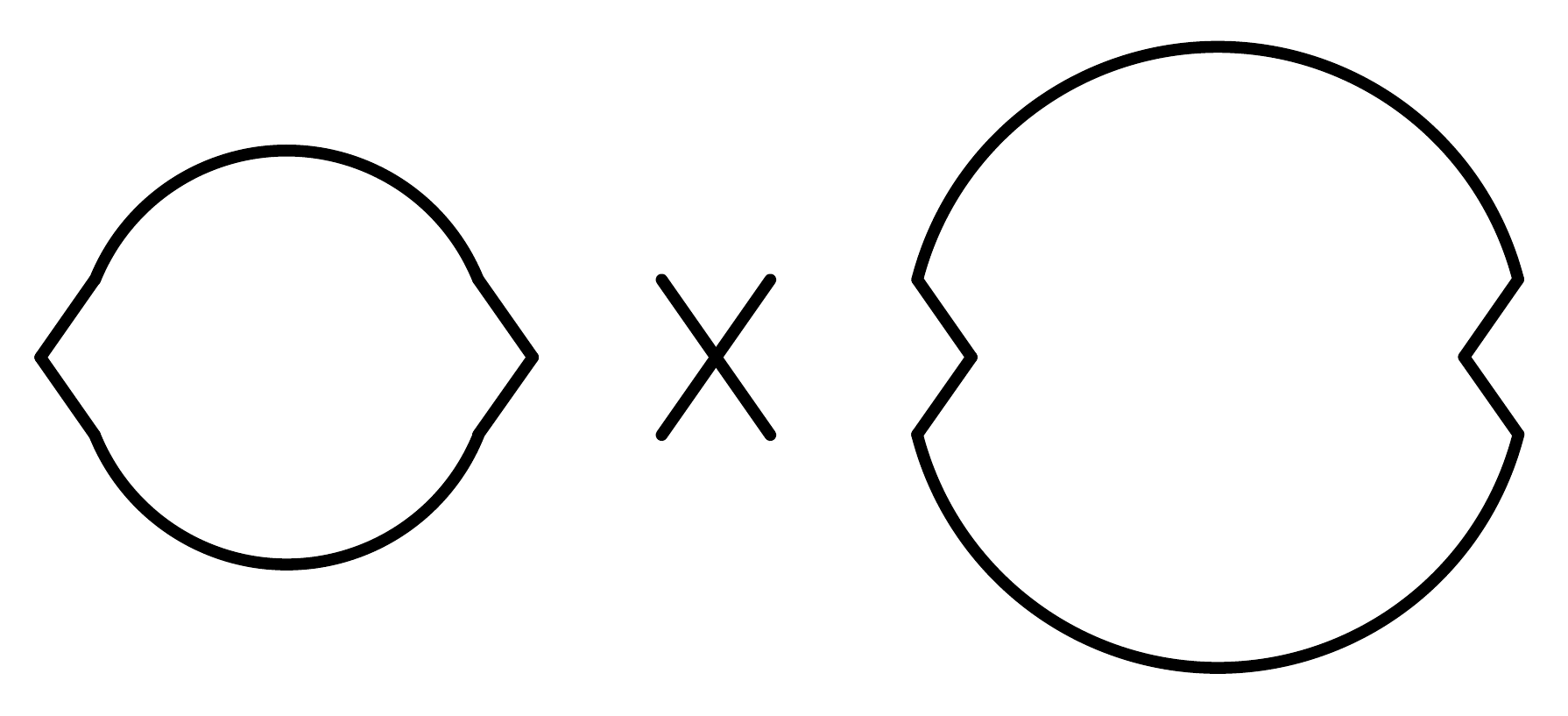}
    \caption{In Example \ref{example:non-2-sided-cad_poincare}, the boundary of $\Omega$ consists of a slightly distorted inner circle (on the left) and a slightly distorted outer circle (on the right) that intersect each other only at two points. However, the cross-like unions of two line segments (one copy in the middle) have ample intersections with both of these distorted circles.}
    \label{figure:annulus_decomposition}
  \end{figure}
  
  As a $1$-dimensional compact Riemannian manifold, a circle supports a weak $1$-Poincar\'e (see \cite[Section 6.1]{heinonenkoskela}). Since both of the distorted circles a bi-Lipschitz equivalent with a regular circle, both of them support a weak $1$-Poincar\'e inequality by Proposition \ref{proposition:bilipschitz_invariance}. As for the unions of two line segments, we first notice that a line segment supports a weak $1$-Poincar\'e inequality because it is bi-Lipschitz equivalent with a piece of the real line (and a connected piece of the real line supports a weak $1$-Poincar\'e inequality by definition and the classical result that the Euclidean space supports a $1$-Poincar\'e inequality (see \cite[Section 8.1]{heinonenetal})). The two line segments in the cross-like union meet only at one point and therefore we cannot use Theorem \ref{theorem:glueing} directly for them. Because of this, we take three line segments and transform them (with bi-Lipschitz mappings) into three V-like shapes which we can then glue one by one to each other with ample intersections to create the original cross-like piece (see Figure \ref{figure:lines_to_cross}).
  
  \begin{figure}[ht]
    \includegraphics[scale=0.9]{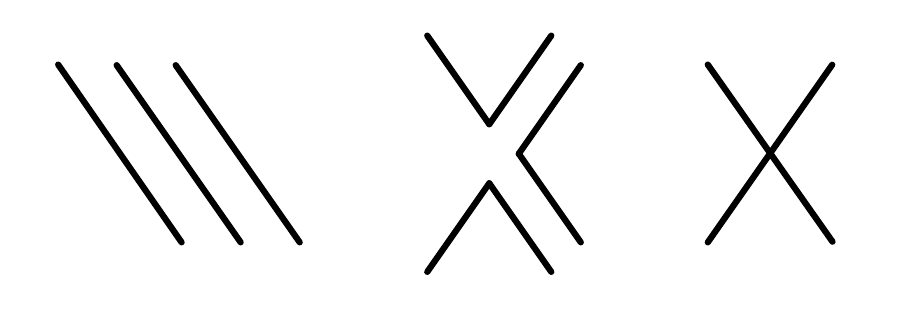}
    \caption{By Theorem \ref{theorem:glueing}, we can preserve weak Poincar\'e inequalities in intersections if the intersections are ample enough. In Example \ref{example:non-2-sided-cad_poincare}, the two intersecting line segments do not have ample intersection but we can create the cross-like shape by glueing three V-like shapes one by one into each other, and we get these V-like shapes by using three line segments and bi-Lipschitz mappings.}
    \label{figure:lines_to_cross}
  \end{figure}
  
  We now get the boundary $\pom$ by glueing first two cross-like pieces to either one of the distorted circles and then glueing the remaining distorted circle to this shape. Glueing sets together using Theorem \ref{theorem:glueing} preserves the weak Poincar\'e inequalities but with a different metric, the one in \eqref{defin:new_metric}. However, for shapes as simple as the ones we use, it is straightforward to see that the new metrics we get are bi-Lipschitz equivalent with the Euclidean metric. Thus, by glueing together the two distorted circles with the help of two cross-like unions of two line segments, we see that $\pom$ supports a weak $1$-Poincar\'e inequality. We note that we had to use two cross-like pieces for the glueing process because with just one cross-like piece the ample intersection requirement of Theorem \ref{theorem:glueing} does not hold for the final step for one of the two intersection points of the distorted circles. By the original, more general form of Theorem \ref{theorem:glueing} in \cite{heinonenkoskela}, we can leave this type of a problematic isolated point out of the glueing process, but this would end up giving us a metric space with a different structure than $\pom$.
\end{example}

\begin{remark}
  \label{remark:higher_dimensions}
  Example \ref{example:non-2-sided-cad_poincare} gives us a $2$-dimensional example but we can use the same techniques to construct higher dimensional examples. These examples are of the type $\Omega \times (0,1)^n$, where $\Omega$ is the $2$-dimensional set from Example \ref{example:non-2-sided-cad_poincare}. Thus, for example, the $3$-dimensional example would be a hollow, twice-pinched cylinder with a thick boundary.
\end{remark}

\begin{remark}
  \label{remark:pinched_annulus}
  By using a once-pinched annulus instead of the twice-pinched one we used in Example \ref{example:non-2-sided-cad_poincare}, we get a connected set $\Omega$ such that it satisfies the $2$-sided corkscrew condition, the boundary $\pom$ is $1$-ADR and $\pom$ supports a weak $1$-Poincar\'e inequality. However, despite connectivity, $\Omega$ is still not a chord-arc domain: there are points arbitrarily close to each other on different sides of the pinched part of the annulus such that they can be connected inside $\Omega$ only by circling around almost the entire annulus.
\end{remark}

Thus, Corollary \ref{corollary:poincare} cannot be reversed and we cannot weaken its assumptions in the obvious way. It is natural to formulate the following problem:

\begin{problem}
  \label{problem:geometric_poincare}
  Let $\Omega \subset \R^{n+1}$ be an open set with $n$-dimensional Ahlfors--David regular boundary. Give a geometric characterization for weak Heinonen--Koskela type boundary $p$-Poincar\'e inequalities for $1 \le p \le n$.
\end{problem}

By a geometric characterization in Problem \ref{problem:geometric_poincare}, we mean a characterization of the type ''$\omega \in \text{weak-}A_\infty(\sigma)$ if and only if $\pom$ is UR and $\Omega$ satisfies the weak local John condition`` (which is the main result of \cite{azzametal1}). By \cite[Theorem 17.1]{cheeger}, we know that $\pom$ has to be quasiconvex, and by \cite{azzam_poincare}, we know that $\pom$ has to be UR. By Corollary \ref{corollary:poincare}, Example \ref{example:non-2-sided-cad_poincare} and \cite[Section 10]{mourgogloutolsa2}, we know that $2$-sided chord-arc domains are too strong for the characterization and ($1$-sided) chord-arc domains are not strong enough for a characterization. However, the answer does not lie somewhere between these two classes of domains: by Example \ref{example:non-2-sided-cad_poincare}, the connectivity properties of $\Omega$ itself do not play a big role in this problem. 

Concerning John-type conditions, we recall the open problem we mentioned in Section \ref{section:john_conditions}:

\begin{problem}
  \label{problem:weak_local_john}
  Let $\Omega \subset \R^{n+1}$ be an open set with $n$-UR (or just $n$-ADR) boundary. Suppose that $\Omega$ satisfies the local John condition. Does $\Omega$ also satisfy the weak local John condition?
\end{problem}
Problem \ref{problem:weak_local_john} is interesting only if $\Omega$ does not satisfy the exterior corkscrew condition: the local John condition implies the corkscrew condition and therefore the answer is trivially true in the presence exterior corkscrews by Lemma \ref{lemma:corkscrews_harmonic_measure}.

\bibliography{localJohn_and_poincare}
\bibliographystyle{alpha}

\end{document}